\documentclass[11pt,draftcls]{IEEEtran}{\onecolumn}
\usepackage{amsthm}
\usepackage{amsmath}
\usepackage{slashbox}
\usepackage{graphicx}
\usepackage{epstopdf}
\usepackage{epsfig,graphics,subfigure,psfrag,amsmath,amssymb}
\usepackage[nooneline,flushleft]{caption2}
\usepackage{cite}
\usepackage{float}
\usepackage{tocvsec2}
\usepackage{color}
\usepackage{verbatim}
\usepackage{algorithm, algorithmic}
\usepackage{mathtools}
\usepackage{hhline}
\usepackage{enumerate}
\usepackage{arydshln}

\newtheorem{thm}{Theorem}
\newtheorem{lem}{Lemma}
\newtheorem{prop}{Proposition}

\newtheorem{defi}{Definition}
\newtheorem{rem}{Remark}

\newtheorem{assump}{Assumption}

\makeatletter

\newcommand{\Rmnum}[1]{\expandafter\@slowromancap\romannumeral #1@}
\makeatother

\begin{document}
\title{Dynamic Sharing Through the ADMM}
\author{Xuanyu Cao and K. J. Ray Liu, \emph{Fellow, IEEE}\\
Email: \{apogne, kjrliu\}@umd.edu\\
Department of Electrical and Computer Engineering, University of Maryland, College Park, MD\\}
\maketitle
\begin{abstract}
In this paper, we study a dynamic version of the sharing problem, in which a dynamic system cost function composed of time-variant local costs of subsystems and a shared time-variant cost of the whole system is minimized. A dynamic alternating direction method of multipliers (ADMM) is proposed to track the varying optimal points of the dynamic optimization problem in an online manner. We analyze the convergence properties of the dynamic ADMM and show that, under several standard technical assumptions, the iterations of the dynamic ADMM converge linearly to some neighborhoods of the time-varying optimal points. The sizes of these neighborhoods depend on the drifts of the dynamic objective functions: the more drastically the dynamic objective function evolves across time, the larger the sizes of these neighborhoods. We also investigate the impact of the drifts on the steady state convergence behaviors of the dynamic ADMM. Finally, two numerical examples, namely a dynamic sharing problem and the dynamic least absolute shrinkage and selection operator (LASSO), are presented to corroborate the effectiveness of the proposed dynamic ADMM. It is observed that the dynamic ADMM can track the time-varying optimal points quickly and accurately. For the dynamic LASSO, the dynamic ADMM has competitive performance compared to the benchmark offline optimizor while the former possesses significant computational advantage over the latter.
\end{abstract}


\section{Introduction}

Many signal processing and resource allocation problems can be posed as an optimization problem which aims at minimizing a system cost consisting of local costs of subsystems and a shared cost of the whole system. For instance, consider a power system divided into multiple subsystems depending on either the geographical locations or the power line connections \cite{gan2013optimal}. On one hand, if a subsystem receives some amount of power supplies, the consumption or storage of these supplies enables the subsystem to gain some utility. On the other hand, the generation of the total power supplies of all the subsystems incurs some cost for the whole power system due to factors such as the consumption of natural resources, the pollution and the human efforts. The goal of the designer or controller of the power system is to maximize the social welfare or minimize the overall system cost including the negative of the total utilities of all the subsystems and the power generation cost of the whole system. This structure of local costs plus shared common cost arises in many applications such as smart grids, communication networks, or more generally, signal processing and control in multi-agent systems. Optimization problems with such structure are called the sharing problems as there is a term of the shared cost of the whole system in the objective function \cite{boyd2011distributed}.

One implicit assumption of the conventional sharing problem is that both the local cost functions and the shared cost function are static, i.e., they do not vary with time. However, in practice, the cost or utility functions of many applications are intrinsically time-varying. For example, in power grids, the utility functions of the subsystems vary across time as the power users' demands evolve, e.g., the demands climax during evening and decline between midnight and early morning. The generation cost of the power system also varies with time owing to the changing and somewhat unpredictable renewable energy sources (e.g., wind and solar energy) as well as the fluctuation of the market prices of the traditional energy. As another example, in real-time signal processing or online learning of multi-agent systems, the data stream arrive sequentially as opposed to in a single batch. This also makes the cost functions involved in the estimators or learners vary with time. Therefore, we are motivated to study a dynamic version of the sharing problem in this paper.

In the literature, dynamic optimization problems arise in various research fields and have been studied from different perspectives. In adaptive signal processing such as the recursive least squares (RLS), the input/output data arrive sequentially, resulting in a time-varying objective function (the discounted total squared errors) to be minimized \cite{Haykin:1996:AFT:230061}. The RLS algorithm is able to track the unknown time-variant weight vectors relating the input and output data in real time. More recently, the concept of adaptive signal processing has been extended to adaptive networks, leading to dynamic distributed optimization problems over networked systems \cite{sayed2014adaptive,jiang2013distributed}. Another line of research for dynamic optimization is online convex optimization (OCO) \cite{hazan2016introduction,hall2015online,mahdavi2012trading,koppel2015saddle,flaxman2005online,zinkevich2003online}. In OCO, the time-varying cost functions are unknown a-priori and the goal is to design online algorithms with low (e.g., sublinear) regrets, i.e., the solution from the algorithms are not too worse than the optimal offline benchmarks. More broadly speaking, online learning (e.g., the weighted majority algorithm and the multiplicative weight update method) \cite{littlestone1989weighted,arora2012multiplicative,kivinen1997exponentiated,tekin2016adaptive,tekin2015distributed,tekin2014distributed} and (stochastic) dynamic control/programming (e.g., Markov decision processes) \cite{puterman2014markov,bertsekas1976dynamic,altman1999constrained} also lie in the category of dynamic optimizations, though their problem formulations are very distinct from that of this paper.

To solve the dynamic sharing problem in an online manner, in this paper, we present a dynamic ADMM algorithm. As a dual domain method, the ADMM is superior to its primal domain counterparts such as the gradient descent method in terms of convergence speed. Due to its broad applicability, the ADMM has been exploited in various signal processing and control problems including distributed estimation \cite{ling2010decentralized}, decentralized optimization \cite{shi2014linear}, wireless communications \cite{shen2012distributed}, power systems \cite{zhang2016admm} and multi-agent coordination \cite{chang2014proximal}. While most existing works only consider static ADMM in which the objectives and constraints are time-invariant, a few recent studies have investigated the ADMM in a dynamic scenario. When the time-varying objective functions are unknown a-priori, an online ADMM algorithm is proposed in \cite{hosseini2014online} to generate solutions with low regrets compared to the optimal static offline solution. This online ADMM is not directly applicable to many dynamic sharing problems in which the goal is to track the time-varying optimal points and a static offline benchmark is not very meaningful. Additionally, several stochastic ADMM algorithms \cite{ouyang2013stochastic,zhong2014fast,suzuki2013dual} have been proposed to solve stochastic programs iteratively using sequential samples. Though the stochastic ADMM operates in a time-varying manner as the new samples arrive sequentially, the underlying stochastic program itself does not change over time, which makes the problem setup very distinct from the dynamic sharing problem considered here. A more closely related work is \cite{ling2014decentralized}, in which a dynamic ADMM algorithm is applied to the consensus optimization problems. However, the convergence analysis of the dynamic ADMM in \cite{ling2014decentralized} significantly relies on the special structure of the consensus optimization problems, in which all agents share the same decision variable. This leaves the performance of the dynamic ADMM in other optimization scenarios largely unknown.

Our goal in this work is to investigate the convergence behaviors of the dynamic ADMM for the dynamic sharing problem both theoretically and empirically. The main contributions of this paper can be summarized as follows.

\begin{itemize}
\item We motivate and formally formulate the dynamic sharing problem (Problem \eqref{dynamic_share}). A dynamic ADMM algorithm (Algorithm 1) is proposed for a more general dynamic optimization problem (Problem \eqref{dynamic_admm}), which encompasses the dynamic sharing problem as a special case. The dynamic ADMM can adapt to the time-varying cost functions and track the optimal points in an online manner.
\item We analyze the convergence properties of the proposed dynamic ADMM algorithm. We show that, under standard technical assumptions, the dynamic ADMM converges linearly to some neighborhoods of the time-varying optimal points. The sizes of the neighborhoods are related to the drifts of the dynamic optimization problem: the more drastically the dynamic problem evolves with time, the larger the sizes of the neighborhoods. We also study the impact of the drifts on the steady state convergence behaviors of the dynamic ADMM.
\item Two numerical examples are presented to validate the effectiveness of the dynamic ADMM algorithm. The first example is a dynamic sharing problem while the second one is the dynamic least absolute shrinkage and selection operator (LASSO). We observe that the dynamic ADMM can track the time-varying optimal points quickly and accurately. For the dynamic LASSO, the dynamic ADMM has competitive performance compared to the benchmark offline optimizor while the former is computationally superior to the latter dramatically.
\end{itemize}

The remaining part of this paper is organized as follows. In Section \Rmnum{2}, the dynamic sharing problem is formally defined and a dynamic ADMM algorithm is proposed. In Section \Rmnum{3}, theoretical analysis of the convergence properties of the dynamic ADMM is presented. Two numerical examples are shown in Section \Rmnum{4}, following which we conclude this work in Section \Rmnum{5}.

\section{Problem Statement and Algorithm Development}

In this section, we first formally state the dynamic sharing problem and give some examples and motivations for it. Then, we present some rudimentary knowledge of the standard static ADMM. Finally, we develop a dynamic ADMM algorithm for a more general dynamic optimization problem, which encompasses the dynamic sharing problem as a special case.

\subsection{The Statement of the Problem}
Consider the sharing problem \cite{boyd2011distributed}:
\begin{align}\label{share}
\text{Minimize}~~\sum_{i=1}^nf^{(i)}\left(\mathbf{x}^{(i)}\right)+g\left(\sum_{i=1}^n\mathbf{x}^{(i)}\right),
\end{align}
with variables $\mathbf{x}^{(i)}\in\mathbb{R}^p$, $i=1,...,n$, where $f^{(i)}:\mathbb{R}^p\mapsto\mathbb{R}$ is the local cost function of subsystem $i$ and $g:\mathbb{R}^p\mapsto\mathbb{R}$ is the global cost function of some commonly shared objective of all subsystems. The global cost function $g$ takes the sum of all $\mathbf{x}^{(i)}$ as its input argument. The sharing problem \eqref{share} is a canonical problem with broad applications in resource allocation and signal processing \cite{boyd2011distributed}. One limitation of the problem formulation in \eqref{share} and its solution methods is that all the cost functions are static, i.e., they do not vary over time. This can be a major obstacle when the application is inherently time-variant and real-time, in which the cost functions change with time and online processing/optimization is imperative. In such circumstances, dynamic algorithms adaptive for the variation of the cost functions are more favorable. This motivates us to study a dynamic version of the sharing problem:
\begin{align}\label{dynamic_share}
\text{Minimize}~~\sum_{i=1}^nf_k^{(i)}\left(\mathbf{x}^{(i)}\right)+g_k\left(\sum_{i=1}^n\mathbf{x}^{(i)}\right),
\end{align}
where $k$ is the time index. $f_k^{(i)}:\mathbb{R}^p\mapsto\mathbb{R}$ is the local cost function of subsystem $i$ at time $k$ and $g_k:\mathbb{R}^p\mapsto\mathbb{R}$ is the global cost function of the shared objective at time $k$. The dynamic sharing problem in \eqref{dynamic_share} can be applied to many dynamic resource allocation problems, among which we name two in the following.
\begin{itemize}
\item Consider a power grid which is divided into $n$ power subsystems according to either geographical locations or power line connections. If subsystem $i$ receives $\mathbf{x}^{(i)}$ amount of power supplies at time $k$, then it gains a utility of $-f_k^{(i)}\left(\mathbf{x}^{(i)}\right)$ by either consuming or storing the supplies. In other words, $f_k^{(i)}$ is the negative of the utility function of power subsystem $i$ at time $k$. The utility function is time-variant because users often have different power demands at different time, e.g., 6-11pm may be the peak demand period while 2-6am may be a low demand period. On the other hand, the generation of the total power supplies of $\sum_{i=1}^n\mathbf{x}^{(i)}$ can incur a cost of $g_k\left(\sum_{i=1}^n\mathbf{x}^{(i)}\right)$ for the power generator due to resource consumptions, human efforts and pollution. The generation cost function $g_k$ also varies across time owing to factors such as the changing and somewhat unpredictable renewable energy sources and the variant prices of the traditional energy sources. Thus, the overall social welfare maximization problem can be posed as a dynamic sharing problem as in \eqref{dynamic_share}.
\item Consider a cognitive radio network composed of $n$ secondary users. If secondary user $i$ obtains spectrum resources $\mathbf{x}^{(i)}$ at time $k$, then her utility is $-f_k^{(i)}\left(\mathbf{x}^{(i)}\right)$. The negative utility function $f_k^{(i)}$ is time-variant because users have different spectrum demands as they change their wireless applications. For example, a user has a high spectrum demand if she is watching online videos. But if she is just making a phone call, her spectrum demand is small. Moreover, for the network operator to provide the total spectrum resources of $\sum_{i=1}^n\mathbf{x}^{(i)}$, he needs to pay a cost of $g_k\left(\sum_{i=1}^n\mathbf{x}^{(i)}\right)$, in which the cost function $g_k$ is also time-varying due to factors including the uncertainty of the sensed temporarily unused spectrum by primary users and the changing market prices of the spectrum. As such, the overall system cost minimization problem can be casted into the form of dynamic sharing problem in \eqref{dynamic_share}.
\end{itemize}

A well-known method to decouple the local cost functions $f^{(i)}$ and the global cost function $g$ in the sharing problem \eqref{share} is the ADMM \cite{boyd2011distributed}. As a dual domain optimization method, the ADMM has faster convergence than the primal domain methods such as the gradient descent algorithm. This inspires us to develop and analyze a dynamic ADMM algorithm to solve the dynamic sharing problem in \eqref{dynamic_share} in this work. Before formal development of the algorithm, we first present a brief review of the basics of ADMM in the next subsection.

\subsection{Preliminaries of ADMM}
ADMM is an optimization framework widely applied to various signal processing applications, including wireless communications \cite{shen2012distributed}, power systems \cite{zhang2016admm} and multi-agent coordination \cite{chang2014proximal}. It enjoys fast convergence speed under mild technical conditions \cite{deng2016global} and is especially suitable for the development of distributed algorithms \cite{boyd2011distributed,bertsekas1989parallel}. ADMM solves problems of the following form:
\begin{eqnarray}\label{admm_prime}
\text{Minimize}_{\mathbf{x},\mathbf{z}} f(\mathbf{x})+g(\mathbf{z})~~\text{s.t.}~~\mathbf{Ax+Bz=c},
\end{eqnarray}
where $\mathbf{A}\in\mathbb{R}^{p\times n},B\in\mathbb{R}^{p\times m},c\in\mathbb{R}^p$ are constants and $\mathbf{x}\in\mathbb{R}^n,\mathbf{z}\in\mathbb{R}^m$ are optimization variables. $f:\mathbb{R}^n\mapsto\mathbb{R}$ and $g:\mathbb{R}^m\mapsto\mathbb{R}$ are two convex functions. The augmented Lagrangian can be formed as:
\begin{equation}
\mathfrak{L}_\rho(\mathbf{x,z,y})=f(\mathbf{x})+g(\mathbf{z})+\mathbf{y}^\mathsf{T}(\mathbf{Ax+Bz-c})+\frac{\rho}{2}\|\mathbf{Ax+Bz-c}\|_2^2,
\end{equation}
where $\mathbf{y}\in\mathbb{R}^p$ is the Lagrange multiplier and $\rho>0$ is some constant. The ADMM then iterates over the following three steps for $k\geq0$ (the iteration index):
\begin{eqnarray}
&&\mathbf{x}^{k+1}=\arg\min_\mathbf{x} \mathfrak{L}_\rho\left(\mathbf{x},\mathbf{z}^k,\mathbf{y}^k\right),\label{x_prime}\\
&&\mathbf{z}^{k+1}=\arg\min_\mathbf{z} \mathfrak{L}_\rho\left(\mathbf{x}^{k+1},\mathbf{z},\mathbf{y}^k\right),\label{z_prime}\\
&&\mathbf{y}^{k+1}=\mathbf{y}^{k}+\rho\left(\mathbf{Ax}^{k+1}+\mathbf{Bz}^{k+1}-\mathbf{c}\right).\label{multiplier_prime}
\end{eqnarray}
The ADMM is guaranteed to converge to the optimal point of \eqref{admm_prime} as long as $f$ and $g$ are convex \cite{boyd2011distributed,bertsekas1989parallel}. It is recently shown that global linear convergence can be ensured provided additional assumptions on problem \eqref{admm_prime} holds \cite{deng2016global}.

\subsection{Development of the Dynamic ADMM}

Define $\mathbf{x}=\left[\mathbf{x}^{(1)\mathsf{T}},...,\mathbf{x}^{(n)\mathsf{T}}\right]^\mathsf{T}\in\mathbb{R}^{np}$, $\mathbf{A}=[\mathbf{I}_p,...,\mathbf{I}_p]\in\mathbb{R}^{p\times np}$ and
\begin{align}\label{f_decompose}
f_k(\mathbf{x})=\sum_{i=1}^nf_k^{(i)}\left(\mathbf{x}^{(i)}\right).
\end{align}
Then, the dynamic sharing problem can be reformulated as:
\begin{align}
\label{dynamic_share_admm}\text{Minimize}_{\mathbf{x}\in\mathbb{R}^{np},\mathbf{z}\in\mathbb{R}^p}~~&f_k(\mathbf{x})+g_k(\mathbf{z})\\
\text{s.t.}~~&\mathbf{Ax-z=0}.
\end{align}
In the remaining part of this paper, we study the following more general dynamic optimization problem:
\begin{align}
\label{dynamic_admm}\text{Minimize}_{\mathbf{x}\in\mathbb{R}^N,\mathbf{z}\in\mathbb{R}^M}~~&f_k(\mathbf{x})+g_k(\mathbf{z})\\
\text{s.t.}~~&\mathbf{Ax+Bz=c},
\end{align}
where $f_k:\mathbb{R}^N\mapsto\mathbb{R}$ and $g_k:\mathbb{R}^{M}\mapsto\mathbb{R}$ are two functions and $\mathbf{A}\in\mathbb{R}^{M\times N},\mathbf{B}\in\mathbb{R}^{M\times M}$ are two matrices. The problem \eqref{dynamic_share_admm} is clearly a special case of the problem \eqref{dynamic_admm} by taking $N=np,M=p,\mathbf{B=-I},\mathbf{c=0}$ and $f_k$ decomposable as in \eqref{f_decompose}. To apply the ADMM, we form the augmented Lagrangian of the problem \eqref{dynamic_admm}:
\begin{align}
\mathfrak{L}_{\rho,k}(\mathbf{x,z},\boldsymbol{\lambda})=f_k(\mathbf{x})+g_k(\mathbf{z})+\boldsymbol{\lambda}^\mathsf{T}(\mathbf{Ax+Bz-c})+\frac{\rho}{2}\|\mathbf{Ax+Bz-c}\|_2^2,
\end{align}
where $\boldsymbol{\lambda}\in\mathbb{R}^M$ is the Lagrange multiplier and $\rho>0$ is some positive constant. Thus, applying the traditional static ADMM \eqref{x_prime}, \eqref{z_prime} and \eqref{multiplier_prime} to the dynamic augmented Lagrangian $\mathfrak{L}_{\rho,k}$, we propose a dynamic ADMM algorithm, as specified in Algorithm 1. The main difference between the dynamic ADMM in Algorithm 1 and the traditional static ADMM described in subsection \Rmnum{2}-B is that the functions $f_k$ and $g_k$ varies across iterations of the ADMM. The aim of this paper is to study the impact of these varying functions on the ADMM algorithm. Lastly, we introduce the following two linear convergence concepts which shall be used later.
\begin{defi}
A sequence $\mathbf{s}_k$ is said to converge \underline{Q-linearly} to $\mathbf{s}^*$ if there exists some constant $\theta\in(0,1)$ such that $\|\mathbf{s}_{k+1}-\mathbf{s}^*\|_2\leq\theta\|\mathbf{s}_k-\mathbf{s}^*\|_2$ for any positive integer $k$.
\end{defi}
\begin{defi}
A sequence $\mathbf{v}_k$ is said to converge \underline{R-linearly} to $\mathbf{v}^*$ if there exists a positive constant $\tau>0$ and some sequence $\mathbf{s}_k$ Q-linearly converging to some point $\mathbf{s}^*$ such that $\|\mathbf{v}_k-\mathbf{v}^*\|_2\leq\tau\|\mathbf{s}_k-\mathbf{s}^*\|_2$ for every positive integer $k$.
\end{defi}

\begin{algorithm}[!htbp]
\renewcommand{\algorithmicrequire}{\textbf{Inputs:} }
\renewcommand\algorithmicensure {\textbf{Outputs:} }
\caption{The dynamic ADMM algorithm for the dynamic problem \eqref{dynamic_admm}}
\begin{algorithmic}[1]
\STATE \texttt{Initialize $\mathbf{x}_0=\mathbf{0},\mathbf{z}_0=\boldsymbol{\lambda}_0=\mathbf{0},k=0$
\STATE \underline{Repeat}:
\STATE $k\leftarrow k+1$
\STATE Update $\mathbf{x}$ according to:
\begin{align}\label{x_update}
\mathbf{x}_k=\arg\min_\mathbf{x}f_k(\mathbf{x})+\boldsymbol{\lambda}_{k-1}^\mathsf{T}\mathbf{Ax}+\frac{\rho}{2}\|\mathbf{Ax+Bz}_{k-1}-\mathbf{c}\|_2^2.
\end{align}
\STATE Update $\mathbf{z}$ according to:
\begin{align}\label{z_update}
\mathbf{z}_k=\arg\min_\mathbf{z}g_k(\mathbf{z})+\boldsymbol{\lambda}_{k-1}^\mathsf{T}\mathbf{Bz}+\frac{\rho}{2}\|\mathbf{Bz+Ax}_k-\mathbf{c}\|_2^2.
\end{align}
\STATE Update $\boldsymbol{\lambda}$ according to:
\begin{align}\label{lambda_update}
\boldsymbol{\lambda}_k=\boldsymbol{\lambda}_{k-1}+\rho(\mathbf{Ax}_k+\mathbf{Bz}_k-\mathbf{c}).
\end{align}
}
\end{algorithmic}
\end{algorithm}

\section{Convergence Analysis}
In this section, convergence analysis for the dynamic ADMM algorithm, i.e., Algorithm 1, is conducted. We first make several standard assumptions for algorithm analysis. Then, we show that the iterations of the dynamic ADMM converge linearly (either Q-linearly or R-linearly) to some neighborhoods of their respective optimal points (Theorem 1 and 2). The sizes of these neighborhoods depend on the \emph{drift} (to be formally defined later) of the dynamic optimization problem \eqref{dynamic_admm}. Finally, we demonstrate the impact of the drift of the dynamic optimization problem \eqref{dynamic_admm} on the steady state convergence properties of the dynamic ADMM.

\subsection{Assumptions}
Throughout the convergence analysis, we make the following assumptions on the functions $f_k$ and $g_k$, all of which are standard in the analysis of optimization algorithms \cite{boyd2004convex,shi2014linear,deng2016global}.
\begin{assump}
For any positive integer $k$, $g_k$ is strongly convex with constant $m>0$ ($m$ is independent of $k$), i.e., for any positive integer $k$:
\begin{align}\label{g_strong}
\left(\nabla g_k(\mathbf{z})-\nabla g_k(\mathbf{z}')\right)^\mathsf{T}(\mathbf{z-z'})\geq m\|\mathbf{z}-\mathbf{z}'\|_2^2,~~\forall\mathbf{z,z'}\in\mathbb{R}^M.
\end{align}
\end{assump}
\begin{assump}
For any positive integer $k$, $f_k$ is strongly convex with constant $\widetilde{m}>0$ ($\widetilde{m}$ is independent of $k$), i.e., for any positive integer $k$:
\begin{align}\label{f_strong}
\left(\nabla f_k(\mathbf{x})-\nabla f_k(\mathbf{x}')\right)^\mathsf{T}(\mathbf{x-x'})\geq \widetilde{m}\|\mathbf{x}-\mathbf{x}'\|_2^2,~~\forall\mathbf{x,x'}\in\mathbb{R}^N.
\end{align}
\end{assump}
\begin{assump}
For any positive integer $k$, $\nabla g_k$ is Lipschitz continuous with constant $L>0$ ($L$ is independent of $k$), i.e., for any positive integer $k$ and any $\mathbf{z,z'}\in\mathbb{R}^M$:
\begin{align}
\|\nabla g_k(\mathbf{z})-\nabla g_k(\mathbf{z}')\|_2\leq L\|\mathbf{z-z}'\|_2.
\end{align}
\end{assump}
We note that, if $g_k$ is twice differentiable, the condition \eqref{g_strong} is equivalent to $\nabla^2g_k(\mathbf{z})\succeq m\mathbf{I},\forall\mathbf{z}$. Similarly, if $f_k$ is twice differentiable, the condition \eqref{f_strong} is equivalent to $\nabla^2f_k(\mathbf{x})\succeq \widetilde{m}\mathbf{I},\forall\mathbf{x}$. This second order definition of strong convexity is more intuitively acceptable and has been used in the analysis of convex optimization algorithms in the literature \cite{boyd2004convex}. But it requires twice differentiability and is not directly useful in the analysis in this work. We further note that when $f_k$ is decomposable as in \eqref{f_decompose} of the dynamic sharing problem, if for any $i=1,...,n$ and positive integer $k$, $f_k^{(i)}$ is strongly convex with constant $\widetilde{m}_i>0$, then Assumption 2 holds with $\widetilde{m}=\min_{i=1,...,n}\widetilde{m}_i>0$. We present the following fact from convex analysis \cite{vandenberghe2016gradient}, which will be used in the later analysis.
\begin{lem}
For any differentiable convex function $h:\mathbb{R}^l\mapsto\mathbb{R}$ and positive constant $L>0$, the following two statements are equivalent:
\begin{enumerate}
\item $\nabla h$ is Lipschitz continuous with constant $L$, i.e., $\left\|\nabla h(\mathbf{x})-\nabla h\left(\mathbf{x}'\right)\right\|_2\leq L\left\|\mathbf{x}-\mathbf{x}'\right\|_2,\forall \mathbf{x},\mathbf{x}'\in\mathbb{R}^l$.
\item $\left\|\nabla h(\mathbf{x})-\nabla h\left(\mathbf{x}'\right)\right\|_2^2\leq L\left(\mathbf{x}-\mathbf{x}'\right)^\mathsf{T}\left(\nabla h(\mathbf{x})-\nabla h\left(\mathbf{x}'\right)\right),\forall\mathbf{x},\mathbf{x}'\in\mathbb{R}^l$.
\end{enumerate}
\end{lem}
We further assume that the matrix $\mathbf{B}\in\mathbb{R}^{M\times M}$ is nonsingular.
\begin{assump}
$\mathbf{B}$ is nonsingular.
\end{assump}

\subsection{Convergence Analysis}

In this subsection, we study the convergence behavior of the proposed dynamic ADMM algorithm under the Assumptions 1-4. Due to the strong convexity assumption in Assumptions 1 and 2, there is a unique primal/dual optimal point pair $(\mathbf{x}_k^*,\mathbf{z}_k^*,\boldsymbol{\lambda}_k^*)$ for the dynamic optimization problem \eqref{dynamic_admm} at time $k$. Denote $\mathbf{u}_k=\left[\mathbf{z}_k^\mathsf{T},\boldsymbol{\lambda}_k^\mathsf{T}\right]^\mathsf{T}$ and $\mathbf{u}_k^*=\left[\mathbf{z}_k^{*\mathsf{T}},\boldsymbol{\lambda}_k^{*\mathsf{T}}\right]^\mathsf{T}$. Since $\mathbf{B}$ is a square matrix, the eigenvalues of $\mathbf{BB}^\mathsf{T}$ are the same as those of $\mathbf{B}^\mathsf{T}\mathbf{B}$. Denote the smallest eigenvalue of $\mathbf{BB}^\mathsf{T}$, which is also the smallest eigenvalue of $\mathbf{B}^\mathsf{T}\mathbf{B}$, as $\alpha$. According to Assumption 4, $\mathbf{B}$ is nonsingular, so $\mathbf{BB}^\mathsf{T}$ and $\mathbf{B}^\mathsf{T}\mathbf{B}$ are positive definite and $\alpha>0$. Define matrix $\mathbf{C}\in\mathbb{R}^{2M\times 2M}$ to be:
\begin{align}
\mathbf{C}=\left[
\begin{array}{cc}
\frac{\rho}{2}\mathbf{B}^\mathsf{T}\mathbf{B}&\\
&\frac{1}{2\rho}\mathbf{I}_M
\end{array}
\right]
\end{align}
Since $\mathbf{B}$ is nonsingular (Assumption 4), we know that $\mathbf{C}$ is positive definite. Therefore, we can define a norm on $\mathbb{R}^{2M}$ as $\|\mathbf{u}\|_\mathbf{C}=\sqrt{\mathbf{u}^\mathsf{T}\mathbf{Cu}}$. Define $t$ to be any arbitrary number within the interval $(0,1)$. A positive constant $\delta>0$ is defined as:
\begin{align}\label{delta_def}
\delta=\min\left\{\frac{2mt}{\rho\|\mathbf{B}\|_2^2},\frac{2\alpha\rho(1-t)}{L}\right\},
\end{align}
where $\|\mathbf{B}\|_2$ is the spectral norm, i.e., the maximum singular value, of $\mathbf{B}$. We further note the following fact, which shall be invoked in later analysis.

\begin{lem}
For any symmetric matrix $\mathbf{A}\in\mathbb{R}^{l\times l}$ and any vectors $\mathbf{x,y,z}\in\mathbb{R}^l$, we have:
\begin{align}
2(\mathbf{x-y})^\mathsf{T}\mathbf{A(z-x)}=\mathbf{(z-y)}^\mathsf{T}\mathbf{A(z-y)}-\mathbf{(x-y)}^\mathsf{T}\mathbf{A(x-y)}-\mathbf{(z-x)}^\mathsf{T}\mathbf{A(z-x)}.
\end{align}
\end{lem}

Now, we are ready to show the first intermediate result.

\begin{prop}
For any positive integer $k$, we have:
\begin{align}\label{intermediate1}
\|\mathbf{u}_k-\mathbf{u}_k^*\|_\mathbf{C}\leq\frac{1}{\sqrt{1+\delta}}\|\mathbf{u}_{k-1}-\mathbf{u}_k^*\|_\mathbf{C}.
\end{align}
\end{prop}

\begin{proof}
Due to the strong convexity assumption in Assumptions 1 and 2, the subproblems in \eqref{x_update} and \eqref{z_update} are unconstrained convex optimization problems. Thus, vanishing gradient is necessary and sufficient for optimality of the subproblems \eqref{x_update} and \eqref{z_update}. The updates of $\mathbf{x}$ and $\mathbf{z}$ can be rewritten as:
\begin{align}
&\label{x_upd}\nabla f_k(\mathbf{x}_k)+\mathbf{A}^\mathsf{T}\boldsymbol{\lambda}_{k-1}+\rho\mathbf{A}^\mathsf{T}(\mathbf{Ax}_k+\mathbf{Bz}_{k-1}-\mathbf{c})=\mathbf{0},\\
&\label{z_upd}\nabla g_k(\mathbf{z}_k)+\mathbf{B}^\mathsf{T}\boldsymbol{\lambda}_{k-1}+\rho\mathbf{B}^\mathsf{T}(\mathbf{Ax}_k+\mathbf{Bz}_k-\mathbf{c})=\mathbf{0}.
\end{align}
Combining \eqref{z_upd} and \eqref{lambda_update} yields:
\begin{align}\label{g_relation}
\nabla g_k(\mathbf{z}_k)+\mathbf{B}^\mathsf{T}\boldsymbol{\lambda}_k=\mathbf{0}.
\end{align}
Combining \eqref{x_upd} and \eqref{lambda_update} gives:
\begin{align}\label{f_relation}
\nabla f_k(\mathbf{x}_k)+\mathbf{A}^\mathsf{T}(\boldsymbol{\lambda}_k+\rho\mathbf{B}(\mathbf{z}_{k-1}-\mathbf{z}_k))=\mathbf{0}.
\end{align}
According to Assumptions 1 and 2, the problem \eqref{dynamic_admm} is a convex optimization problem. Thus, Karush-Kuhn-Tucker (KKT) conditions are necessary and sufficient for optimality. Hence,
\begin{align}
\label{f_kkt}&\nabla f_k(\mathbf{x}_k^*)+\mathbf{A}^\mathsf{T}\boldsymbol{\lambda}_k^*=\mathbf{0},\\
\label{g_kkt}&\nabla g_k(\mathbf{z}_k^*)+\mathbf{B}^\mathsf{T}\boldsymbol{\lambda}_k^*=\mathbf{0},\\
\label{feasible_kkt}&\mathbf{Ax}_k^*+\mathbf{Bz}_k^*=\mathbf{c}.
\end{align}
Because of the convexity of $g_k$ (Assumption 1) and Lipschitz continuity of its gradient (Assumption 3), by invoking Lemma 1, we get:
\begin{align}
\|\nabla g_k(\mathbf{z}_k)-\nabla g_k(\mathbf{z}_k^*)\|_2^2\leq L(\mathbf{z}_k-\mathbf{z}_k^*)^\mathsf{T}(\nabla g_k(\mathbf{z}_k)-\nabla g_k(\mathbf{z}_k^*)).
\end{align}
Further using \eqref{g_kkt} and \eqref{g_relation}, we obtain:
\begin{align}\label{1}
(\mathbf{z}_k-\mathbf{z}_k^*)^\mathsf{T}\mathbf{B}^\mathsf{T}(\boldsymbol{\lambda}_k^*-\boldsymbol{\lambda}_k)\geq\frac{1}{L}\left\|\mathbf{B}^\mathsf{T}(\boldsymbol{\lambda}_k-\boldsymbol{\lambda}_k^*)\right\|_2^2.
\end{align}
According to the strong convexity of $g_k$ (Assumption 1), we have:
\begin{align}\label{2}
m\|\mathbf{z}_k-\mathbf{z}_k^*\|_2^2\leq(\nabla g_k(\mathbf{z}_k)-\nabla g_k(\mathbf{z}_k^*))^\mathsf{T}(\mathbf{z}_k-\mathbf{z}_k^*)=\left(-\mathbf{B}^\mathsf{T}\boldsymbol{\lambda}_k+\mathbf{B}^\mathsf{T}\boldsymbol{\lambda}_k^*\right)^\mathsf{T}(\mathbf{z}_k-\mathbf{z}_k^*).
\end{align}
Combining \eqref{1} and \eqref{2}, we know that for any $t\in(0,1)$:
\begin{align}\label{3}
(\mathbf{z}_k-\mathbf{z}_k^*)^\mathsf{T}\mathbf{B}^\mathsf{T}(\boldsymbol{\lambda}_k^*-\boldsymbol{\lambda}_k)\geq tm\|\mathbf{z}_k-\mathbf{z}_k^*\|_2^2+\frac{1-t}{L}\left\|\mathbf{B}^\mathsf{T}(\boldsymbol{\lambda}_k-\boldsymbol{\lambda}_k^*)\right\|_2^2.
\end{align}
According to the convexity of $f_k$ (Assumption 2), we have:
\begin{align}
0\leq (\mathbf{x}_k-\mathbf{x}_k^*)^\mathsf{T}(\nabla f_k(\mathbf{x}_k)-\nabla f_k(\mathbf{x}_k^*)).
\end{align}
Further making use of \eqref{f_relation} and \eqref{f_kkt}, we get:
\begin{align}\label{4}
0\leq(\mathbf{x}_k-\mathbf{x}_k^*)^\mathsf{T}\mathbf{A}^\mathsf{T}(\boldsymbol{\lambda}_k^*-\boldsymbol{\lambda}_k+\rho\mathbf{B}(\mathbf{z}_k-\mathbf{z}_{k-1})).
\end{align}
Adding \eqref{3} and \eqref{4} leads to:
\begin{align}
\begin{split}\label{6}
&(\mathbf{z}_k-\mathbf{z}_k^*)^\mathsf{T}\mathbf{B}^\mathsf{T}(\boldsymbol{\lambda}_k^*-\boldsymbol{\lambda}_k)+(\mathbf{x}_k-\mathbf{x}_k^*)^\mathsf{T}\mathbf{A}^\mathsf{T}(\boldsymbol{\lambda}_k^*-\boldsymbol{\lambda}_k+\rho\mathbf{B}(\mathbf{z}_k-\mathbf{z}_{k-1}))\\
&\geq tm\|\mathbf{z}_k-\mathbf{z}_k^*\|_2^2+\frac{1-t}{L}\left\|\mathbf{B}^\mathsf{T}(\boldsymbol{\lambda}_k-\boldsymbol{\lambda}_k^*)\right\|_2^2.
\end{split}
\end{align}
From \eqref{lambda_update} and \eqref{feasible_kkt}, we get:
\begin{align}\label{5}
\mathbf{A}(\mathbf{x}_k-\mathbf{x}_k^*)+\mathbf{B}(\mathbf{z}_k-\mathbf{z}_k^*)=\frac{1}{\rho}(\boldsymbol{\lambda}_k-\boldsymbol{\lambda}_{k-1}).
\end{align}
Making use of \eqref{5}, we derive:
\begin{align}
&(\mathbf{z}_k-\mathbf{z}_k^*)^\mathsf{T}\mathbf{B}^\mathsf{T}(\boldsymbol{\lambda}_k^*-\boldsymbol{\lambda}_k)+(\mathbf{x}_k-\mathbf{x}_k^*)^\mathsf{T}\mathbf{A}^\mathsf{T}(\boldsymbol{\lambda}_k^*-\boldsymbol{\lambda}_k+\rho\mathbf{B}(\mathbf{z}_k-\mathbf{z}_{k-1}))\\
&=(\boldsymbol{\lambda}_k^*-\boldsymbol{\lambda}_k)^\mathsf{T}[\mathbf{B}(\mathbf{z}_k-\mathbf{z}_k^*)+\mathbf{A}(\mathbf{x}_k-\mathbf{x}_k^*)]+\rho(\mathbf{x}_k-\mathbf{x}_k^*)^\mathsf{T}\mathbf{A}^\mathsf{T}\mathbf{B}(\mathbf{z}_k-\mathbf{z}_{k-1})\\
&=\frac{1}{\rho}(\boldsymbol{\lambda}_k^*-\boldsymbol{\lambda}_k)^\mathsf{T}(\boldsymbol{\lambda}_k-\boldsymbol{\lambda}_{k-1})+\rho(\mathbf{A}(\mathbf{x}_k-\mathbf{x}_k^*))^\mathsf{T}\mathbf{B}(\mathbf{z}_k-\mathbf{z}_{k-1})\\
&=\frac{1}{\rho}(\boldsymbol{\lambda}_{k-1}-\boldsymbol{\lambda}_k)^\mathsf{T}(\boldsymbol{\lambda}_k-\boldsymbol{\lambda}_k^*)+(\boldsymbol{\lambda}_k-\boldsymbol{\lambda}_{k-1}-\rho\mathbf{B}(\mathbf{z}_k-\mathbf{z}_k^*))^\mathsf{T}\mathbf{B}(\mathbf{z}_k-\mathbf{z}_{k-1})
\end{align}
Together with \eqref{6}, we get:
\begin{align}
&\frac{1}{\rho}(\boldsymbol{\lambda}_{k-1}-\boldsymbol{\lambda}_k)^\mathsf{T}(\boldsymbol{\lambda}_k-\boldsymbol{\lambda}_k^*)+\rho(\mathbf{z}_k-\mathbf{z}_k^*)^\mathsf{T}\mathbf{B}^\mathsf{T}\mathbf{B}(\mathbf{z}_{k-1}-\mathbf{z}_k)\\
&\geq (\boldsymbol{\lambda}_k-\boldsymbol{\lambda}_{k-1})^\mathsf{T}\mathbf{B}(\mathbf{z}_{k-1}-\mathbf{z}_k)+tm\|\mathbf{z}_k-\mathbf{z}_k^*\|_2^2+\frac{1-t}{L}\left\|\mathbf{B}^\mathsf{T}(\boldsymbol{\lambda}_k-\boldsymbol{\lambda}_k^*)\right\|_2^2,
\end{align}
which is equivalent to:
\begin{align}
&\frac{1}{\rho}(\boldsymbol{\lambda}_{k-1}-\boldsymbol{\lambda}_k)^\mathsf{T}(\boldsymbol{\lambda}_k-\boldsymbol{\lambda}_{k-1}+\boldsymbol{\lambda}_{k-1}-\boldsymbol{\lambda}_k^*)+\rho(\mathbf{z}_{k-1}-\mathbf{z}_k)^\mathsf{T}\mathbf{B}^\mathsf{T}\mathbf{B}(\mathbf{z}_k-\mathbf{z}_{k-1}+\mathbf{z}_{k-1}-\mathbf{z}_k^*)\\
&\geq (\boldsymbol{\lambda}_k-\boldsymbol{\lambda}_{k-1})^\mathsf{T}\mathbf{B}(\mathbf{z}_{k-1}-\mathbf{z}_k)+tm\|\mathbf{z}_k-\mathbf{z}_k^*\|_2^2+\frac{1-t}{L}\left\|\mathbf{B}^\mathsf{T}(\boldsymbol{\lambda}_k-\boldsymbol{\lambda}_k^*)\right\|_2^2.
\end{align}
This can be further rewritten as:
\begin{align}
\begin{split}\label{9}
&\frac{1}{\rho}(\boldsymbol{\lambda}_{k-1}-\boldsymbol{\lambda}_k)^\mathsf{T}(\boldsymbol{\lambda}_{k-1}-\boldsymbol{\lambda}_k^*)+\rho(\mathbf{z}_{k-1}-\mathbf{z}_k)^\mathsf{T}\mathbf{B}^\mathsf{T}\mathbf{B}(\mathbf{z}_{k-1}-\mathbf{z}_k^*)\\
&\geq \frac{1}{\rho}\|\boldsymbol{\lambda}_{k-1}-\boldsymbol{\lambda}_k\|_2^2+\rho\|\mathbf{Bz}_{k-1}-\mathbf{Bz}_k\|_2^2+(\boldsymbol{\lambda}_k-\boldsymbol{\lambda}_{k-1})^\mathsf{T}\mathbf{B}(\mathbf{z}_{k-1}-\mathbf{z}_k)\\
&~~~+mt\|\mathbf{z}_k-\mathbf{z}_k^*\|_2^2+\frac{1-t}{L}\left\|\mathbf{B}^\mathsf{T}(\boldsymbol{\lambda}_k-\boldsymbol{\lambda}_k^*)\right\|_2^2.
\end{split}
\end{align}

Making use of Lemma 2, we obtain:
\begin{align}
\begin{split}\label{7}
&\frac{1}{\rho}(\boldsymbol{\lambda}_{k-1}-\boldsymbol{\lambda}_k)^\mathsf{T}(\boldsymbol{\lambda}_{k-1}-\boldsymbol{\lambda}_k^*)\\
&=-\frac{1}{2\rho}\|\boldsymbol{\lambda}_k^*-\boldsymbol{\lambda}_k\|_2^2+\frac{1}{2\rho}\|\boldsymbol{\lambda}_{k-1}-\boldsymbol{\lambda}_k\|_2^2+\frac{1}{2\rho}\|\boldsymbol{\lambda}_k^*-\boldsymbol{\lambda}_{k-1}\|_2^2,
\end{split}
\end{align}
and
\begin{align}
\begin{split}\label{8}
&\rho(\mathbf{z}_{k-1}-\mathbf{z}_k)^\mathsf{T}\mathbf{B}^\mathsf{T}\mathbf{B}(\mathbf{z}_{k-1}-\mathbf{z}_k^*)\\
&=-\frac{\rho}{2}\|\mathbf{Bz}_k^*-\mathbf{Bz}_k\|_2^2+\frac{\rho}{2}\|\mathbf{Bz}_{k-1}-\mathbf{Bz}_k\|_2^2+\frac{\rho}{2}\|\mathbf{Bz}_k^*-\mathbf{Bz}_{k-1}\|_2^2.
\end{split}
\end{align}
Combining \eqref{7} and \eqref{8} and further utilizing \eqref{9} gives:
\begin{align}
&\frac{1}{2\rho}\|\boldsymbol{\lambda}_{k-1}-\boldsymbol{\lambda}_k^*\|_2^2+\frac{\rho}{2}\|\mathbf{Bz}_{k-1}-\mathbf{Bz}_k^*\|_2^2-\frac{1}{2\rho}\|\boldsymbol{\lambda}_k-\boldsymbol{\lambda}_k^*\|_2^2-\frac{\rho}{2}\|\mathbf{Bz}_k-\mathbf{Bz}_k^*\|_2^2\\
&=\frac{1}{\rho}(\boldsymbol{\lambda}_{k-1}-\boldsymbol{\lambda}_k)^\mathsf{T}(\boldsymbol{\lambda}_{k-1}-\boldsymbol{\lambda}_k^*)+\rho(\mathbf{z}_{k-1}-\mathbf{z}_k)^\mathsf{T}\mathbf{B}^\mathsf{T}\mathbf{B}(\mathbf{z}_{k-1}-\mathbf{z}_k^*)\nonumber\\
&~~~-\frac{1}{2\rho}\|\boldsymbol{\lambda}_{k-1}-\boldsymbol{\lambda}_k\|_2^2-\frac{\rho}{2}\|\mathbf{Bz}_{k-1}-\mathbf{Bz}_k\|_2^2\\
&\geq \frac{1}{2\rho}\|\boldsymbol{\lambda}_{k-1}-\boldsymbol{\lambda}_k\|_2^2+\frac{\rho}{2}\|\mathbf{Bz}_{k-1}-\mathbf{Bz}_k\|_2^2+(\boldsymbol{\lambda}_k-\boldsymbol{\lambda}_{k-1})^\mathsf{T}\mathbf{B}(\mathbf{z}_{k-1}-\mathbf{z}_k)\nonumber\\
&~~~+mt\|\mathbf{z}_k-\mathbf{z}_k^*\|_2^2+\frac{1-t}{L}\left\|\mathbf{B}^\mathsf{T}(\boldsymbol{\lambda}_k-\boldsymbol{\lambda}_k^*)\right\|_2^2\\
&=\frac{1}{2\rho}\left\|\boldsymbol{\lambda}_k-\boldsymbol{\lambda}_{k-1}+\rho(\mathbf{Bz}_{k-1}-\mathbf{Bz}_k)\right\|_2^2+mt\|\mathbf{z}_k-\mathbf{z}_k^*\|_2^2+\frac{1-t}{L}\left\|\mathbf{B}^\mathsf{T}(\boldsymbol{\lambda}_k-\boldsymbol{\lambda}_k^*)\right\|_2^2\\
&\geq mt\|\mathbf{z}_k-\mathbf{z}_k^*\|_2^2+\frac{1-t}{L}\left\|\mathbf{B}^\mathsf{T}(\boldsymbol{\lambda}_k-\boldsymbol{\lambda}_k^*)\right\|_2^2\\
&\geq\frac{mt}{\|\mathbf{B}\|_2^2}\|\mathbf{Bz}_k-\mathbf{Bz}_k^*\|_2^2+\frac{\alpha(1-t)}{L}\|\boldsymbol{\lambda}_k-\boldsymbol{\lambda}_k^*\|_2^2\\
&\geq\delta\left(\frac{1}{2\rho}\|\boldsymbol{\lambda}_k-\boldsymbol{\lambda}_k^*\|_2^2+\frac{\rho}{2}\|\mathbf{Bz}_k-\mathbf{Bz}_k^*\|_2^2\right),
\end{align}
where the last step is due to the definition of $\delta$ in \eqref{delta_def}. Noticing the definition of $\|\cdot\|_\mathbf{C}$, we get:
\begin{align}
\|\mathbf{u}_{k-1}-\mathbf{u}_k^*\|_\mathbf{C}^2\geq(1+\delta)\|\mathbf{u}_k-\mathbf{u}_k^*\|_\mathbf{C}^2,
\end{align}
which is tantamount to \eqref{intermediate1}.
\end{proof}

\begin{rem}
Proposition 1 states that $\mathbf{u}_k$ is closer to $\mathbf{u}_k^*$ than $\mathbf{u}_{k-1}$ with a shrinkage factor of $\delta$. The bigger the $\delta$, the stronger the shrinkage. Note that there is an arbitrary factor $t\in(0,1)$ in the definition of $\delta$ in \eqref{delta_def}. By choosing $t=\frac{\alpha\rho^2\|\mathbf{B}\|_2^2}{mL+\alpha\rho^2\|\mathbf{B}\|_2^2}$, we get the maximum $\delta$ as $\delta_{\max}=\frac{2m\alpha\rho}{mL+\alpha\rho^2\|\mathbf{B}\|_2^2}$. In the expression of $\delta_{\max}$, only $\rho$ is a tunable algorithm parameter while all other parameters are given by the optimization problem. The fact that $\delta_{\max}$ increases with $\rho$ may partially justify the need of a relatively large $\rho$ for good convergence behaviors of the dynamic ADMM. We will investigate the impact of $\rho$ on algorithm performance empirically in Section \Rmnum{4}.
\end{rem}

Proposition 1 establishes a relation between $\|\mathbf{u}_k-\mathbf{u}_k^*\|_\mathbf{C}$ and $\|\mathbf{u}_{k-1}-\mathbf{u}_k^*\|_\mathbf{C}$. However, to describe the convergence behavior of the dynamic ADMM algorithm, what we really want is the relation between $\|\mathbf{u}_k-\mathbf{u}_k^*\|_\mathbf{C}$ and $\|\mathbf{u}_{k-1}-\mathbf{u}_{k-1}^*\|_\mathbf{C}$. This is accomplished by the following theorem.

\begin{thm}
Define the drift $d_k$ of the dynamic problem \eqref{dynamic_admm} to be:
\begin{align}\label{drift_def}
d_k=\sqrt{\frac{\rho}{2}}\|\mathbf{B}\|_2\|\mathbf{z}_{k-1}^*-\mathbf{z}_k^*\|_2+\frac{1}{\sqrt{2\rho\alpha}}\|\nabla g_{k-1}(\mathbf{z}_{k-1}^*)-\nabla g_k(\mathbf{z}_k^*)\|_2.
\end{align}
Then, for any integer $k\geq2$, we have:
\begin{align}
\|\mathbf{u}_k-\mathbf{u}_k^*\|_\mathbf{C}\leq\frac{1}{\sqrt{1+\delta}}(\|\mathbf{u}_{k-1}-\mathbf{u}_{k-1}^*\|_\mathbf{C}+d_k).
\end{align}
\end{thm}
\begin{proof}
According to \eqref{g_kkt}, we have:
\begin{align}
\label{a1}&\nabla g_k(\mathbf{z}_k^*)+\mathbf{B}^\mathsf{T}\boldsymbol{\lambda}_k^*=\mathbf{0},\\
\label{a2}&\nabla g_{k-1}(\mathbf{z}_{k-1}^*)+\mathbf{B}^\mathsf{T}\boldsymbol{\lambda}_{k-1}^*=\mathbf{0}.
\end{align}
Substraction of \eqref{a2} from \eqref{a1} yields:
\begin{align}
\mathbf{B}^\mathsf{T}(\boldsymbol{\lambda}_k^*-\boldsymbol{\lambda}_{k-1}^*)=-\nabla g_k(\mathbf{z}_k^*)+\nabla g_{k-1}(\mathbf{z}_{k-1}^*).
\end{align}
Hence,
\begin{align}
&\|\nabla g_{k-1}(\mathbf{z}_{k-1}^*)-\nabla g_k(\mathbf{z}_k^*)\|_2^2\\
&=\|\mathbf{B}^\mathsf{T}(\boldsymbol{\lambda}_{k-1}^*-\boldsymbol{\lambda}_k^*)\|_2^2\\
&=(\boldsymbol{\lambda}_{k-1}^*-\boldsymbol{\lambda}_k^*)^\mathsf{T}\mathbf{B}\mathbf{B}^\mathsf{T}(\boldsymbol{\lambda}_{k-1}^*-\boldsymbol{\lambda}_k^*)\\
\label{a3}&\geq\alpha\|\boldsymbol{\lambda}_{k-1}^*-\boldsymbol{\lambda}_k^*\|_2^2.
\end{align}
On the other hand,
\begin{align}\label{a4}
(\mathbf{z}_{k-1}^*-\mathbf{z}_k^*)^\mathsf{T}\mathbf{B}^\mathsf{T}\mathbf{B}(\mathbf{z}_{k-1}^*-\mathbf{z}_k^*)\leq\|\mathbf{B}\|_2^2\|\mathbf{z}_{k-1}^*-\mathbf{z}_k^*\|_2^2.
\end{align}
Combining \eqref{a3} and \eqref{a4}, we get:
\begin{align}
&\|\mathbf{u}_{k-1}^*-\mathbf{u}_k^*\|_\mathbf{C}^2\\
&=\frac{\rho}{2}(\mathbf{z}_{k-1}^*-\mathbf{z}_k^*)^\mathsf{T}\mathbf{B}^\mathsf{T}\mathbf{B}(\mathbf{z}_{k-1}^*-\mathbf{z}_k^*)+\frac{1}{2\rho}\|\boldsymbol{\lambda}_{k-1}^*-\boldsymbol{\lambda}_k^*\|_2^2\\
&\leq\frac{\rho}{2}\|\mathbf{B}\|_2^2\|\mathbf{z}_{k-1}^*-\mathbf{z}_k^*\|_2^2+\frac{1}{2\rho\alpha}\|\nabla g_{k-1}(\mathbf{z}_{k-1}^*)-\nabla g_k(\mathbf{z}_k^*)\|_2^2\\
&\leq\left(\sqrt{\frac{\rho}{2}}\|\mathbf{B}\|_2\|\mathbf{z}_{k-1}^*-\mathbf{z}_k^*\|_2+\frac{1}{\sqrt{2\rho\alpha}}\|\nabla g_{k-1}(\mathbf{z}_{k-1}^*)-\nabla g_k(\mathbf{z}_k^*)\|_2\right)^2\\
&=d_k^2.
\end{align}
Thus, $\|\mathbf{u}_{k-1}^*-\mathbf{u}_k^*\|_\mathbf{C}\leq d_k$ and:
\begin{align}\label{a5}
\|\mathbf{u}_{k-1}-\mathbf{u}_k^*\|_\mathbf{C}\leq\|\mathbf{u}_{k-1}-\mathbf{u}_{k-1}^*\|_\mathbf{C}+\|\mathbf{u}_{k-1}^*-\mathbf{u}_k^*\|_\mathbf{C}\leq\|\mathbf{u}_{k-1}-\mathbf{u}_{k-1}^*\|_\mathbf{C}+d_k.
\end{align}
Combining \eqref{a5} and \eqref{intermediate1} in Proposition 1 gives:
\begin{align}
\|\mathbf{u}_k-\mathbf{u}_k^*\|_\mathbf{C}\leq\frac{1}{\sqrt{1+\delta}}(\|\mathbf{u}_{k-1}-\mathbf{u}_{k-1}^*\|_\mathbf{C}+d_k).
\end{align}
\end{proof}
\begin{rem}
Theorem 1 means that $\mathbf{u}_k$ converges Q-linearly (with contraction factor $\sqrt{1+\delta}$) to some neighborhood of the optimal point $\mathbf{u}_k^*$. The size of the neighborhood is characterized by $d_k$, the drift of the dynamic problem \eqref{dynamic_admm}, which is determined by the problem formulation instead of the algorithm. The more drastically the dynamic problem \eqref{dynamic_admm} varies across time, the bigger the drift $d_k$, and the larger the size of that neighborhood. When the dynamic problem \eqref{dynamic_admm} degenerates to its static counterpart, i.e., $f_k$ and $g_k$ does not vary with $k$, $d_k$ becomes zero. In such a case, Theorem 1 degenerates to the linear convergence result of static ADMM in \cite{deng2016global}.
\end{rem}
Q-linear convergence of $\mathbf{u}_k$ to some neighborhood of the optimal point $\mathbf{u}_k^*$ is established in Theorem 1. A more meaningful result will be about the convergence properties of $\mathbf{x}_k,\mathbf{z}_k,\boldsymbol{\lambda}_k$. To this end, we want to link the quantities $\|\mathbf{x}_k-\mathbf{x}_k^*\|_2$, $\|\mathbf{z}_k-\mathbf{z}_k^*\|_2$, $\|\boldsymbol{\lambda}_k-\boldsymbol{\lambda}_k^*\|_2$ with $\|\mathbf{u}_k-\mathbf{u}_k^*\|_\mathbf{C}$. This is accomplished by the following theorem.
\begin{thm}
For any integer $k\geq2$, we have:
\begin{align}
\begin{split}\label{a8}
&\|\mathbf{x}_k-\mathbf{x}_k^*\|_2\\
&\leq\frac{1}{\widetilde{m}}\|\mathbf{A}\|_2\left[\left(\sqrt{2\rho}+\|\mathbf{B}\|_2\sqrt{\frac{2\rho}{\alpha}}\right)\|\mathbf{u}_k-\mathbf{u}_k^*\|_\mathbf{C}+\|\mathbf{B}\|_2\sqrt{\frac{2\rho}{\alpha}}\|\mathbf{u}_{k-1}-\mathbf{u}_{k-1}^*\|_\mathbf{C}+\sqrt{2\rho}d_k\right],
\end{split}
\end{align}
where $\|\mathbf{A}\|_2$ is the spectral norm, i.e., the largest singular value, of $\mathbf{A}$.
Furthermore, for any positive integer $k$, we have:
\begin{align}
\label{a6}&\|\mathbf{z}_k-\mathbf{z}_k^*\|_2\leq\sqrt{\frac{2}{\alpha\rho}}\|\mathbf{u}_k-\mathbf{u}_k^*\|_\mathbf{C},\\
\label{a7}&\|\boldsymbol{\lambda}_k-\boldsymbol{\lambda}_k^*\|_2\leq\sqrt{2\rho}\|\mathbf{u}_k-\mathbf{u}_k^*\|_\mathbf{C}.
\end{align}
\end{thm}
\begin{proof}
\eqref{a6} and \eqref{a7} are straightforward. According to the definition of $\|\cdot\|_\mathbf{C}$, we have $\|\mathbf{u}_k-\mathbf{u}_k^*\|_\mathbf{C}^2\geq\frac{1}{2\rho}\|\boldsymbol{\lambda}_k-\boldsymbol{\lambda}_k^*\|_2^2$, which results in \eqref{a7}. Moreover, we note:
\begin{align}
\|\mathbf{Bz}_k-\mathbf{Bz}_k^*\|_2^2=(\mathbf{z}_k-\mathbf{z}_k^*)^\mathsf{T}\mathbf{B}^\mathsf{T}\mathbf{B}(\mathbf{z}_k-\mathbf{z}_k^*)\geq\alpha\|\mathbf{z}_k-\mathbf{z}_k^*\|_2^2,
\end{align}
and
\begin{align}
\|\mathbf{u}_k-\mathbf{u}_k^*\|_\mathbf{C}^2\geq\frac{\rho}{2}\|\mathbf{Bz}_k-\mathbf{Bz}_k^*\|_2^2,
\end{align}
which together lead to \eqref{a6}. Now, we proceed to prove \eqref{a8}. From the strong convexity of $f_k$ (Assumption 2) and equations \eqref{f_relation}, \eqref{f_kkt}, we derive:
\begin{align}
&\widetilde{m}\|\mathbf{x}_k-\mathbf{x}_k^*\|_2^2\\
&\leq(\mathbf{x}_k-\mathbf{x}_k^*)^\mathsf{T}(\nabla f_k(\mathbf{x}_k)-\nabla f_k(\mathbf{x}_k^*))\\
&=(\mathbf{x}_k-\mathbf{x}_k^*)^\mathsf{T}\left[\mathbf{A}^\mathsf{T}(-\boldsymbol{\lambda}_k+\rho\mathbf{B}(\mathbf{z}_k-\mathbf{z}_{k-1}))+\mathbf{A}^\mathsf{T}\boldsymbol{\lambda}_k^*\right]\\
&=(\mathbf{x}_k-\mathbf{x}_k^*)^\mathsf{T}\mathbf{A}^\mathsf{T}[\boldsymbol{\lambda}_k^*-\boldsymbol{\lambda}_k+\rho\mathbf{B}(\mathbf{z}_k-\mathbf{z}_{k-1})]\\
&\leq\|\mathbf{x}_k-\mathbf{x}_k^*\|_2\|\mathbf{A}\|_2(\|\boldsymbol{\lambda}_k-\boldsymbol{\lambda}_k^*\|_2+\rho\|\mathbf{B}\|_2\|\mathbf{z}_k-\mathbf{z}_{k-1}\|_2).
\end{align}
Therefore,
\begin{align}
&\|\mathbf{x}_k-\mathbf{x}_k^*\|_2\\
&\leq\frac{1}{\widetilde{m}}\|\mathbf{A}\|_2(\|\boldsymbol{\lambda}_k-\boldsymbol{\lambda}_k^*\|_2+\rho\|\mathbf{B}\|_2\|\mathbf{z}_k-\mathbf{z}_{k-1}\|_2)\\
&\leq\frac{1}{\widetilde{m}}\|\mathbf{A}\|_2[\|\boldsymbol{\lambda}_k-\boldsymbol{\lambda}_k^*\|_2+\rho\|\mathbf{B}\|_2(\|\mathbf{z}_k-\mathbf{z}_k^*\|_2+\|\mathbf{z}_k^*-\mathbf{z}_{k-1}^*\|_2+\|\mathbf{z}_{k-1}-\mathbf{z}_{k-1}^*\|_2)]
\end{align}
Further exploiting \eqref{a6} (for $k$ and $k-1$) and \eqref{a7}, we obtain:
\begin{align}
\begin{split}\label{a9}
&\|\mathbf{x}_k-\mathbf{x}_k^*\|_2\\
&\leq\frac{1}{\widetilde{m}}\|\mathbf{A}\|_2\Bigg[\left(\sqrt{2\rho}+\|\mathbf{B}\|_2\sqrt{\frac{2\rho}{\alpha}}\right)\|\mathbf{u}_k-\mathbf{u}_k^*\|_\mathbf{C}+\|\mathbf{B}\|_2\sqrt{\frac{2\rho}{\alpha}}\|\mathbf{u}_{k-1}-\mathbf{u}_{k-1}^*\|_\mathbf{C}\\
&~~~~+\rho\|\mathbf{B}\|_2\|\mathbf{z}_k^*-\mathbf{z}_{k-1}^*\|_2\Bigg]
\end{split}
\end{align}
According to the definition of the drift $d_k$ in \eqref{drift_def}, we know that $\|\mathbf{z}_k^*-\mathbf{z}_{k-1}^*\|_2\leq\frac{1}{\|\mathbf{B}\|_2}\sqrt{\frac{2}{\rho}}d_k$. Substituting this relation into \eqref{a9} leads to \eqref{a8}.
\end{proof}
\begin{rem}
Sine $\mathbf{u}_k$ converges Q-linearly to some neighborhood of $\mathbf{u}_k^*$ (Theorem 1), Theorem 2 indicates that $\mathbf{x}_k,\mathbf{z}_k,\boldsymbol{\lambda}_k$ converge R-linearly to some neighborhoods of $\mathbf{x}_k^*,\mathbf{z}_k^*,\boldsymbol{\lambda}_k^*$, respectively. When the dynamic optimization problem \eqref{dynamic_admm} degenerates to its static version, i.e., $f_k$ and $g_k$ does not vary with $k$, Theorem 2 also degenerates to its static counterpart in \cite{deng2016global,shi2014linear}.
\end{rem}
To see the impact of the drift $d_k$ (and thus the difference between the dynamic ADMM and the static ADMM) on the steady state convergence behaviors, we present the following result.
\begin{thm}
Suppose the drift defined in \eqref{drift_def} satisfies $d_k\leq d,\forall k$, for some $d\in\mathbb{R}$. Then, we have:
\begin{align}
\label{b1}&\limsup_{k\rightarrow\infty}\|\mathbf{u}_k-\mathbf{u}_k^*\|_\mathbf{C}\leq\frac{d}{\sqrt{1+\delta}-1},\\
\label{b2}&\limsup_{k\rightarrow\infty}\|\mathbf{x}_k-\mathbf{x}_k^*\|_2\leq\frac{\|\mathbf{A}\|_2}{\widetilde{m}}\left[\frac{\sqrt{2\rho}+\|\mathbf{B}\|_2\sqrt{\frac{8\rho}{\alpha}}}{\sqrt{1+\delta}-1}+\sqrt{2\rho}\right]d,\\
\label{b3}&\limsup_{k\rightarrow\infty}\|\mathbf{z}_k-\mathbf{z}_k^*\|_2\leq\sqrt{\frac{2}{\alpha\rho}}\frac{d}{\sqrt{1+\delta}-1},\\
\label{b4}&\limsup_{k\rightarrow\infty}\|\boldsymbol{\lambda}_k-\boldsymbol{\lambda}_k^*\|_2\leq\sqrt{2\rho}\frac{d}{\sqrt{1+\delta}-1}.
\end{align}
\end{thm}
\begin{proof}
According to Theorem 1, we have:
\begin{align}
&\left(\sqrt{1+\delta}\right)^k\|\mathbf{u}_k-\mathbf{u}_k^*\|_\mathbf{C}\leq\left(\sqrt{1+\delta}\right)^{k-1}\|\mathbf{u}_{k-1}-\mathbf{u}_{k-1}^*\|_\mathbf{C}+\left(\sqrt{1+\delta}\right)^{k-1}d_k,\\
&\left(\sqrt{1+\delta}\right)^{k-1}\|\mathbf{u}_{k-1}-\mathbf{u}_{k-1}^*\|_\mathbf{C}\leq\left(\sqrt{1+\delta}\right)^{k-2}\|\mathbf{u}_{k-2}-\mathbf{u}_{k-2}^*\|_\mathbf{C}+\left(\sqrt{1+\delta}\right)^{k-2}d_{k-1},\\
&\cdots\cdots\nonumber\\
&\left(\sqrt{1+\delta}\right)^2\|\mathbf{u}_2-\mathbf{u}_2^*\|_\mathbf{C}\leq\sqrt{1+\delta}\|\mathbf{u}_1-\mathbf{u}_1^*\|_\mathbf{C}+\sqrt{1+\delta}d_2.
\end{align}
Summing them together gives:
\begin{align}
&\left(\sqrt{1+\delta}\right)^k\|\mathbf{u}_k-\mathbf{u}_k^*\|_\mathbf{C}\\
&\leq\sum_{i=1}^{k-1}\left(\sqrt{1+\delta}\right)^id_{i+1}+\sqrt{1+\delta}\|\mathbf{u}_1-\mathbf{u}_1^*\|_\mathbf{C}\\
&\leq d\sqrt{1+\delta}\sum_{i=0}^{k-2}\left(\sqrt{1+\delta}\right)^i+\sqrt{1+\delta}\|\mathbf{u}_1-\mathbf{u}_1^*\|_\mathbf{C}.
\end{align}
Hence,
\begin{align}
\|\mathbf{u}_k-\mathbf{u}_k^*\|_\mathbf{C}\leq\frac{d}{\sqrt{1+\delta}-1}+\frac{\|\mathbf{u}_1-\mathbf{u}_1^*\|_\mathbf{C}}{\left(\sqrt{1+\delta}\right)^{k-1}},
\end{align}
which results in \eqref{b1}. Combining \eqref{b1} with \eqref{a6} and \eqref{a7} immediately leads to \eqref{b3} and \eqref{b4}. In addition, according to \eqref{a8}, we have:
\begin{align}
\begin{split}
&\|\mathbf{x}_k-\mathbf{x}_k^*\|_2\\
&\leq\frac{1}{\widetilde{m}}\|\mathbf{A}\|_2\left[\left(\sqrt{2\rho}+\|\mathbf{B}\|_2\sqrt{\frac{2\rho}{\alpha}}\right)\|\mathbf{u}_k-\mathbf{u}_k^*\|_\mathbf{C}+\|\mathbf{B}\|_2\sqrt{\frac{2\rho}{\alpha}}\|\mathbf{u}_{k-1}-\mathbf{u}_{k-1}^*\|_\mathbf{C}+\sqrt{2\rho}d\right].
\end{split}
\end{align}
Therefore,
\begin{align}
&\limsup_{k\rightarrow\infty}\|\mathbf{x}_k-\mathbf{x}_k^*\|_2\\
&\leq\frac{\|\mathbf{A}\|_2}{\widetilde{m}}\left[\left(\sqrt{2\rho}+\|\mathbf{B}\|_2\sqrt{\frac{8\rho}{\alpha}}\right)\limsup_{k\rightarrow\infty}\|\mathbf{u}_k-\mathbf{u}_k^*\|_\mathbf{C}+\sqrt{2\rho}d\right]\\
&\leq\frac{\|\mathbf{A}\|_2}{\widetilde{m}}\left[\frac{\sqrt{2\rho}+\|\mathbf{B}\|_2\sqrt{\frac{8\rho}{\alpha}}}{\sqrt{1+\delta}-1}+\sqrt{2\rho}\right]d.
\end{align}
\end{proof}

\section{Numerical Examples}
In this section, two numerical examples are presented to validate the effectiveness of the proposed dynamic ADMM algorithm, Algorithm 1. The first example is a dynamic sharing problem and the second one is the dynamic least absolute shrinkage and selection operator (LASSO). Through these two examples, we confirm that the proposed dynamic ADMM algorithm is suitable for not only the dynamic sharing problem \eqref{dynamic_share} (e.g., the first numerical example) but also the general form of dynamic optimization problem \eqref{dynamic_admm} (e.g., the second example, dynamic LASSO, which is not a sharing problem).

\subsection{The Dynamic Sharing Problem}
\subsubsection{Problem Formulation and Algorithm Development}
We first consider the following dynamic sharing problem:
\begin{align}\label{numerical_sharing}
\text{Minimize}_{\mathbf{x}^{(1)},...,\mathbf{x}^{(n)}\in\mathbb{R}^p}~~\sum_{i=1}^n\left(\mathbf{x}^{(i)}-\boldsymbol{\theta}_k^{(i)}\right)^\mathsf{T}\mathbf{\Phi}_k^{(i)}\left(\mathbf{x}^{(i)}-\boldsymbol{\theta}_k^{(i)}\right)+\gamma\left\|\sum_{i=1}^n\mathbf{x}^{(i)}\right\|_1,
\end{align}
where $\boldsymbol{\theta}_k^{(i)}\in\mathbb{R}^p$, $\mathbf{\Phi}_k^{(i)}\in\mathbb{R}^{p\times p}$ positive definite, $\gamma>0$ are given problem data. A motivating application instance of the problem \eqref{numerical_sharing} can be as follows. Suppose there are $n$ subsystems and $p$ quantities (such as data flow in communication networks or currents in power grids) distributed over these subsystems. The amount of the $p$ quantities at subsystem $i$ is described by the vector $\mathbf{x}_i\in\mathbb{R}^p$. Our goal is to estimate the vectors $\mathbf{x}_i,i=1,2,...,n$. The problem data at time $k$ are $\mathbf{\Phi}_k^{(i)},\boldsymbol{\theta}_k^{(i)}$, which vary across time as we keep obtaining new measurements and updating the problem data. We assume that the statistical model of the vectors $\mathbf{x}_i$ is Gaussian so that the first term in \eqref{numerical_sharing} corresponds to the negative log likelihood. Suppose the sum of most quantities across all subsystems cancel out (such as the generation/consumption of power due to the energy conservation rule and the incoming/outgoing current or data flow due to the Kirchhoff's laws) while a few do not cancel out because of abnormality such as leakage. This implies that the sum $\sum_{i=1}^n\mathbf{x}^{(i)}$ should be sparse, i.e., most entries are zero. To incorporate this prior knowledge of sparsity into the estimator, we introduce the $l_1$ regularization term, i.e., the second term in \eqref{numerical_sharing}. Therefore, the estimator is tantamount to the dynamic sharing problem in \eqref{numerical_sharing}.

The problem \eqref{numerical_sharing} is clearly in the form of \eqref{dynamic_share} with:
\begin{align}
&f_k^{(i)}\left(\mathbf{x}^{(i)}\right)=\left(\mathbf{x}^{(i)}-\boldsymbol{\theta}_k^{(i)}\right)^\mathsf{T}\mathbf{\Phi}_k^{(i)}\left(\mathbf{x}^{(i)}-\boldsymbol{\theta}_k^{(i)}\right),\\
&g_k(\mathbf{z})=\gamma\|\mathbf{z}\|_1.
\end{align}
Define:
\begin{align}
\mathbf{x}=\left[
\begin{array}{c}
\mathbf{x}^{(1)}\\
\mathbf{x}^{(2)}\\
\vdots\\
\mathbf{x}^{(n)}
\end{array}
\right],~~\boldsymbol{\theta}_k=\left[
\begin{array}{c}
\boldsymbol{\theta}_k^{(1)}\\
\boldsymbol{\theta}_k^{(2)}\\
\vdots\\
\boldsymbol{\theta}_k^{(n)}
\end{array}
\right],~~\mathbf{\Phi}_k=\left[
\begin{array}{cccc}
\mathbf{\Phi}_k^{(1)}&&&\\
&\mathbf{\Phi}_k^{(2)}&&\\
&&\ddots&\\
&&&\mathbf{\Phi}_k^{(n)}
\end{array}
\right].
\end{align}
Thus, in terms of problem \eqref{dynamic_share_admm}, we have:
\begin{align}
f_k(\mathbf{x})=\left(\mathbf{x}-\boldsymbol{\theta}_k\right)^\mathsf{T}\mathbf{\Phi}_k\left(\mathbf{x}-\boldsymbol{\theta}_k\right).
\end{align}
Applying the dynamic ADMM algorithm, i.e., Algorithm 1, to this dynamic sharing problem, we obtain Algorithm \ref{num_share}. The soft-threshold function $\mathcal{S}$ is defined for $a\in\mathbb{R},\kappa>0$ as follows:
\begin{align}
\mathcal{S}_\kappa(a)=
\begin{cases}
a-\kappa,\text{~~if~~}a>\kappa,\\
0,\text{~~if~~}|a|\leq\kappa,\\
a+\kappa,\text{~~if~~}a<\kappa.
\end{cases}
\end{align}
In \eqref{soft}, an entrywise extension of the soft-threshold function to vector input is used.

\begin{algorithm}[!htbp]
\caption{The dynamic ADMM algorithm for the dynamic sharing problem \eqref{numerical_sharing}}
\begin{algorithmic}[1]\label{num_share}
\STATE \texttt{Initialize $\mathbf{x}_0=\mathbf{0},\mathbf{z}_0=\boldsymbol{\lambda}_0=\mathbf{0},k=0$
\STATE \underline{Repeat}:
\STATE $k\leftarrow k+1$
\STATE Update $\mathbf{x}$ according to:
\begin{align}
\mathbf{x}_k=\left(2\mathbf{\Phi}_k+\rho\mathbf{A}^\mathsf{T}\mathbf{A}\right)^{-1}\left(2\mathbf{\Phi}_k\boldsymbol{\theta}_k-\mathbf{A}^\mathsf{T}\boldsymbol{\lambda}_{k-1}+\rho\mathbf{A}^\mathsf{T}\mathbf{z}_{k-1}\right).
\end{align}
\STATE Update $\mathbf{z}$ according to:
\begin{align}\label{soft}
\mathbf{z}_k=\mathcal{S}_\frac{\gamma}{\rho}\left(\mathbf{Ax}_k+\frac{\boldsymbol{\lambda}_{k-1}}{\rho}\right).
\end{align}
\STATE Update $\boldsymbol{\lambda}$ according to:
\begin{align}
\boldsymbol{\lambda}_k=\boldsymbol{\lambda}_{k-1}+\rho(\mathbf{Ax}_k-\mathbf{z}_k).
\end{align}
}
\end{algorithmic}
\end{algorithm}

\subsubsection{Generation of $\mathbf{\Phi}_k^{(i)}$ and $\boldsymbol{\theta}_k^{(i)}$}

We generate the problem data $\mathbf{\Phi}_k^{(i)}$ and $\boldsymbol{\theta}_k^{(i)}$ recursively as follows. Given $\mathbf{\Phi}_{k-1}^{(i)}~(k\geq 1)$, we first generate $\widetilde{\mathbf{\Phi}}_k^{(i)}$ according to $\widetilde{\mathbf{\Phi}}_k^{(i)}=\mathbf{\Phi}_{k-1}^{(i)}+\eta_k^{(i)}\mathbf{E}_k^{(i)}$, where $\eta_k^{(i)}$ is some small positive number and $\mathbf{E}_k^{(i)}$ is a random symmetric matrix with entries uniformly distributed on $[-1,1]$. Then, we construct the matrix $\mathbf{\Phi}_k^{(i)}$ as:
\begin{align}
\mathbf{\Phi}_k^{(i)}=
\begin{cases}
\widetilde{\Phi}_k^{(i)},~~\text{if}~~\lambda_{\min}\left(\widetilde{\mathbf{\Phi}}_k^{(i)}\right)\geq\epsilon,~\text{i.e.,}~\widetilde{\mathbf{\Phi}}_k^{(i)}\succeq\epsilon\mathbf{I},\\
\widetilde{\Phi}_k^{(i)}+\left[\epsilon-\lambda_{\min}\left(\widetilde{\mathbf{\Phi}}_k^{(i)}\right)\right]\mathbf{I},~~\text{otherwise},
\end{cases}
\end{align}
where $\lambda_{\min}(\cdot)$ denotes the smallest eigenvalue and $\epsilon>0$ is some positive constant. Through this construction, we ensure that $\mathbf{\Phi}_k^{(i)}\succeq\epsilon\mathbf{I},k=1,2,...$. In addition, $\mathbf{\Phi}_0$ is a random symmetric matrix whose entries are uniformly distributed on $[-1,1]$.

Given $\boldsymbol{\theta}_{k-1}^{(i)}~(k\geq1)$, we generate $\boldsymbol{\theta}_k^{(i)}$ according to:
\begin{align}
\boldsymbol{\theta}_k^{(i)}=\boldsymbol{\theta}_{k-1}^{(i)}+\eta_k^{(i)}\mathbf{h}_k^{(i)},
\end{align}
where $\mathbf{h}_k^{(i)}$ is a random $p$-dimensional vector whose entries are uniformly distributed on $[-1,1]$. $\boldsymbol{\theta}_0^{(i)}$ is also a random $p$-dimensional vector with entries uniformly distributed on $[-1,1]$.

\subsubsection{Simulation Results}

\begin{figure}
  \centering
  \includegraphics[scale=.3]{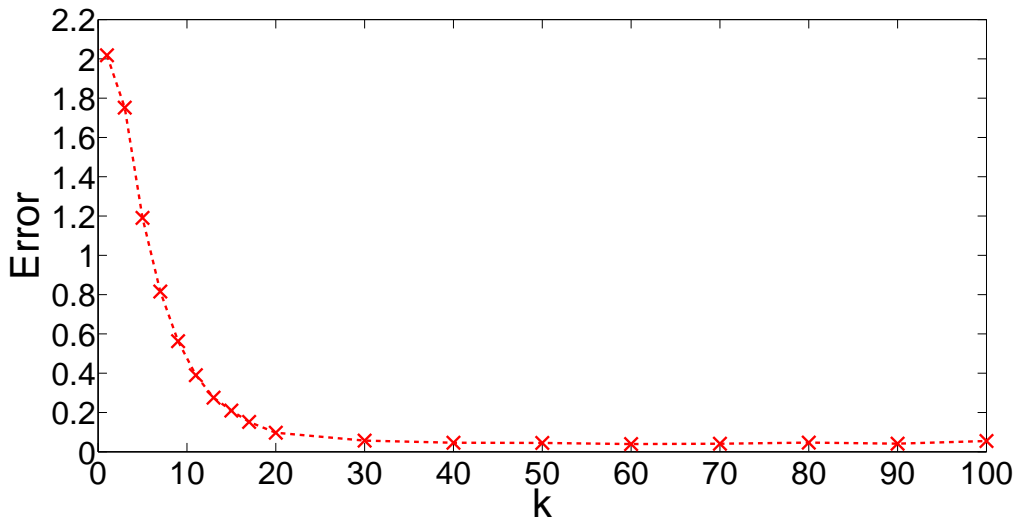}\\
  \caption{The convergence curve of $\|\mathbf{x}_k-\mathbf{x}_k^*\|_2$. $\mathbf{x}_k^*$ is the optimal point of the dynamic sharing problem \eqref{numerical_sharing} at time $k$ computed by an offline optimizor. $\mathbf{x}_k$ is the online solution given by the proposed dynamic ADMM algorithm, i.e., Algorithm \ref{num_share}.}\label{resource}
\end{figure}

In the first simulation, we set the parameters as $\eta=0.2,\epsilon=1,\gamma=1,\rho=1,p=5,n=20$. We use the CVX package \cite{cvx,gb08} to compute the optimal point $\mathbf{x}_k^*$ of the instance of the dynamic sharing problem \eqref{numerical_sharing} at time $k$ in an offline manner. The convergence curve of $\|\mathbf{x}_k-\mathbf{x}_k^*\|_2$ ($\mathbf{x}_k$ is the online solution given by the proposed dynamic ADMM algorithm, i.e., Algorithm \ref{num_share}) is shown in Fig. \ref{resource}. The result is the average of 100 independent trials. We observe that $\mathbf{x}_k$ can converge to some neighborhood of $\mathbf{x}_k^*$ after about 30 iterations. This corroborates the theoretical results (Theorem 2 and Theorem 3) and the effectiveness of the proposed dynamic ADMM algorithm.

\begin{figure}
  \centering
  \includegraphics[scale=.3]{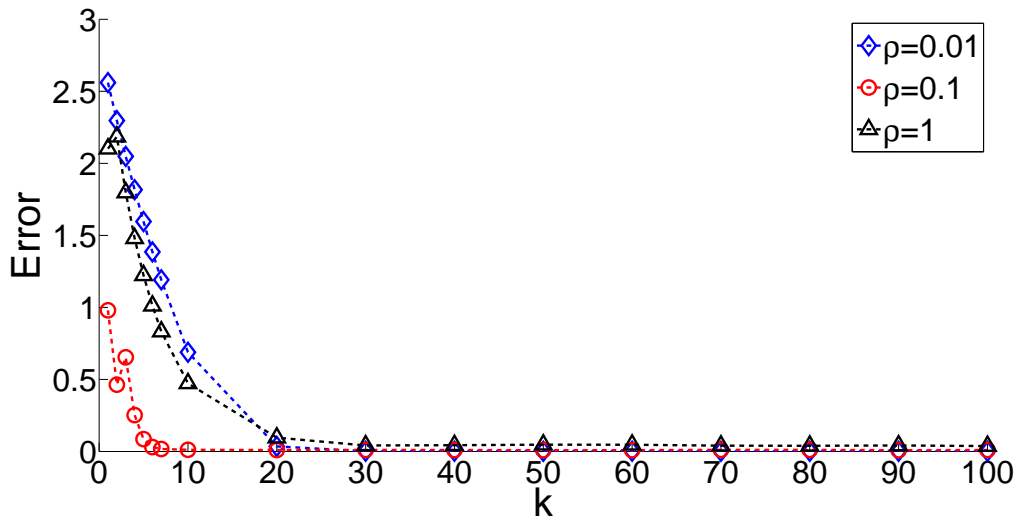}\\
  \caption{The impact of the algorithm parameter $\rho$ on the convergence behaviors ($\|\mathbf{x}_k-\mathbf{x}_k^*\|_2$) of the dynamic ADMM.}\label{resource_rho}
\end{figure}

In the second simulation, we investigate the impact of the algorithm parameter $\rho$ on the convergence performance of the dynamic ADMM. We consider three different values for $\rho$: $0.01,0.1,1$. The corresponding convergence curves ($\|\mathbf{x}_k-\mathbf{x}_k^*\|_2$) are shown in Fig. \ref{resource_rho}. We find that $\rho=0.1$ yields the best convergence performance among the three circumstances. This indicates that the importance of an appropriate value of $\rho$, which should be neither too large nor too small. We note that similar observations have been made in the traditional static ADMM \cite{boyd2011distributed}.

\subsection{Dynamic LASSO}

\subsubsection{Problem Formulation}
Least absolute shrinkage and selection operator (LASSO) is an important and renowned problem in statistics and signal processing. It embodies sparsity-aware linear regression. Here, we consider a dynamic version of the LASSO since the problem data often vary with time in many real-time applications as new measurements arrive sequentially:
\begin{align}\label{lasso_num}
\text{Minimize}_{\mathbf{x}\in\mathbb{R}^p}~~\frac{1}{2}\|\mathbf{F}_k\mathbf{x}-\mathbf{h}_k\|_2^2+\gamma\|\mathbf{x}\|_1,
\end{align}
where $\mathbf{F}_k\in\mathbb{R}^{m\times p}$, $\mathbf{h}_k\in\mathbb{R}^m$ are time-variant problem data and $\gamma>0$ is some positive constant controlling the sparsity of the solution. The problem \eqref{lasso_num} is clearly in the form of \eqref{dynamic_admm}  with $f_k(\mathbf{x})=\frac{1}{2}\|\mathbf{F}_k\mathbf{x}-\mathbf{h}_k\|_2^2$, $g_k(\mathbf{z})=\gamma\|\mathbf{z}\|_1$, $\mathbf{A}=\mathbf{I}$, $\mathbf{B}=-\mathbf{I}$, $\mathbf{c=0}$. Thus, we can apply Algorithm 1 to the problem \eqref{lasso_num}, where both \eqref{x_update} and \eqref{z_update} admit closed-form solutions. Note that the problem \eqref{lasso_num} does not fall into the category of dynamic sharing problem \eqref{dynamic_share} as $f_k(\mathbf{x})$ cannot be decomposed across several parts of $\mathbf{x}$. Our goal in this numerical example is to show that the proposed dynamics algorithm works well for the general dynamic optimization problem \eqref{dynamic_admm}, not just the dynamic sharing problem.

\subsubsection{Generation of $\mathbf{F}_k$ and $\mathbf{h}_k$}
The problem data $\mathbf{F}_k$ and $\mathbf{h}_k$ are generated as follows. Given $\mathbf{F}_{k-1}~(k\geq1)$, we generate $\mathbf{F}_k$ according to:
\begin{align}
\mathbf{F}_k=\mathbf{F}_{k-1}+\eta_k\mathbf{W}_k,
\end{align}
where $\eta_k$ is some small positive constant and $\mathbf{W}_k\in\mathbb{R}^{m\times p}$ is a random matrix with entries uniformly distributed on $[-1,1]$. $\mathbf{F}_0$ is also a random matrix with entries uniformly distributed on $[-1,1]$.

To generate the sequence $\mathbf{h}_k$, we construct an auxiliary ground-truth sequence $\widetilde{\mathbf{x}}_k$ as follows. We randomly select $q$ different numbers $\{j_1,...j_q\}$ from the set $\{1,...,p\}$, where $q\ll p$. Given $\widetilde{\mathbf{x}}_{k-1}~(k\geq1)$, we generate $\widetilde{\mathbf{x}}_k$ based on:
\begin{align}
\widetilde{\mathbf{x}}_k=\widetilde{\mathbf{x}}_{k-1}+\eta_k\mathbf{u}_k,
\end{align}
where $\mathbf{u}_k\in\mathbb{R}^p$ is a random vector with $j_l$-th entry uniformly distributed on $[-1,1]$, $l=1,...,q$ and other entries equal to zero. $\widetilde{\mathbf{x}}_0$ is a random vector whose $j_l$-th entry is uniformly distributed on $[0,1]$, $l=1,...,q$ and other entries are zero. This enforces sparsity of the ground-truth $\widetilde{\mathbf{x}}_k$ to be estimated, which is the underlying hypothesis of the LASSO. With $\widetilde{\mathbf{x}}_k$ and $\mathbf{F}_k$ in hands, we generate $\mathbf{h}_k$ according to:
\begin{align}
\mathbf{h}_k=\mathbf{F}_k\widetilde{\mathbf{x}}_k+\mathbf{v}_k,
\end{align}
where $\mathbf{v}_k\sim\mathcal{N}(\mathbf{0},\sigma^2\mathbf{I})$ is a $m$-dimensional Gaussian random vector.

\subsubsection{Simulation Results}

\begin{figure}
\renewcommand\figurename{\small Fig.}
\centering \vspace*{8pt} \setlength{\baselineskip}{10pt}

\subfigure[Slowly time-variant case, i.e., $\eta=0.01$.]{
\includegraphics[scale = 0.3]{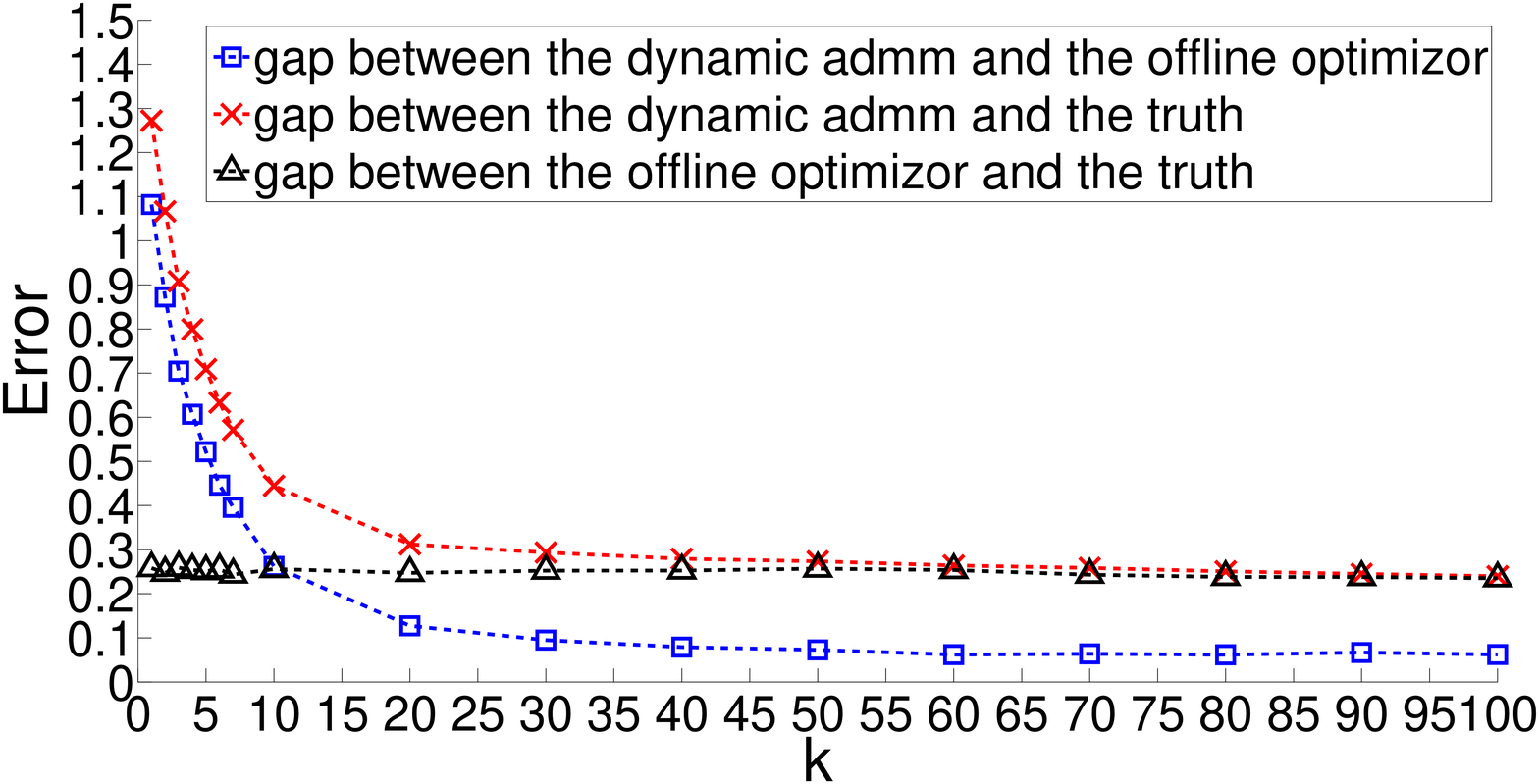}}
\subfigure[Fast time-variant case, i.e., $\eta=0.1$.]{
\includegraphics[scale = 0.3]{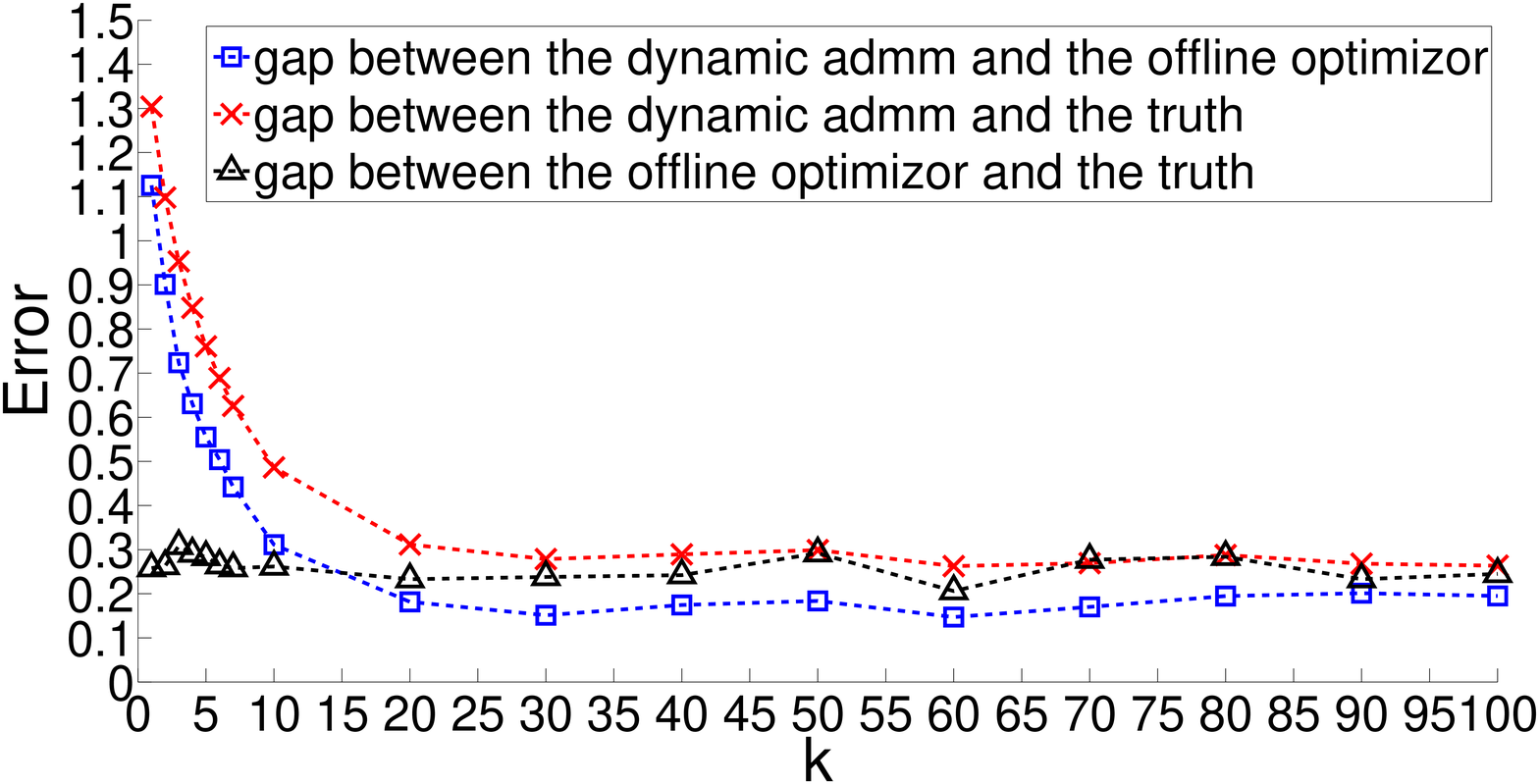}}

\caption{The gaps between the online estimate generated by applying the dynamic ADMM to the dynamic LASSO \eqref{lasso_num}, the estimate given by the offline optimizor through the CVX package (i.e., the optimal point of \eqref{lasso_num}) and the ground-truth: $\|\mathbf{x}_k-\mathbf{x}_k^*\|_2,\|\mathbf{x}_k-\widetilde{\mathbf{x}}_k\|_2$ and $\|\mathbf{x}_k^*-\widetilde{\mathbf{x}}_k\|_2$.}
\label{lasso_error}
\end{figure}

In the simulations, we set the parameters as: $m=10,p=30,q=2,\rho=1,\gamma=0.2,\sigma=0.1$. All results except Fig. \ref{lasso_track} are the average of 100 independent trials. We consider two values, 0.01 and 0.1, for $\eta$, the parameter controlling the variation of the problem data across time. We call $\eta=0.01$ and $\eta=0.1$ the slowly time-variant case and the fast time-variant case, respectively. Denote the online estimate generated by applying the dynamic ADMM to the dynamic LASSO \eqref{lasso_num}, the estimate given by the offline optimizor through the CVX package (i.e., the optimal point of \eqref{lasso_num}) and the ground-truth as $\mathbf{x}_k,\mathbf{x}^*_k$ and $\widetilde{\mathbf{x}}_k$, respectively. The gaps between these three quantities, i.e., $\|\mathbf{x}_k-\mathbf{x}_k^*\|_2,\|\mathbf{x}_k-\widetilde{\mathbf{x}}_k\|_2$ and $\|\mathbf{x}_k^*-\widetilde{\mathbf{x}}_k\|_2$, in the slowly time-variant case and the fast time-variant case are reported in Fig. \ref{lasso_error}-(a) and Fig. \ref{lasso_error}-(b), respectively. A few remarks are in order. First, the solution of the optimizor $\mathbf{x}_k^*$ should be regarded as the benchmark for the dynamic ADMM as the former is the optimal point of \eqref{lasso_num}, or in other words, the best that the dynamic LASSO can achieve. For both slowly and fast time-variant cases, the gaps between the dynamic ADMM and the offline optimizor, i.e., the blue line with square marker, converge to some small values after about 40 iterations. This indicates that the dynamic ADMM can track the optimal point of \eqref{lasso_num} well. Second, the gaps between the dynamic ADMM and the truth (red line with cross markers) as well as the gaps between the offline optimizor and the truth (black line with triangle markers) are similar after some 50 iterations in both slowly and fast time-variant cases. This suggests that in terms of tracking the ground-truth, the dynamic ADMM and the offline optimizor have similar performances while the former has much less computational complexity than the latter. Third, unsurprisingly, comparing \ref{lasso_error}-(a) with \ref{lasso_error}-(b), we observe that the tracking performances of both the dynamic ADMM and the offline optimizor are related to the value of $\eta$: the larger the $\eta$, the more drastically the change of the problem data across time, the poorer the tracking performance.

\begin{figure}
\renewcommand\figurename{\small Fig.}
\centering \vspace*{8pt} \setlength{\baselineskip}{10pt}

\subfigure[Slowly time-variant case, i.e., $\eta=0.01$.]{
\includegraphics[scale = 0.3]{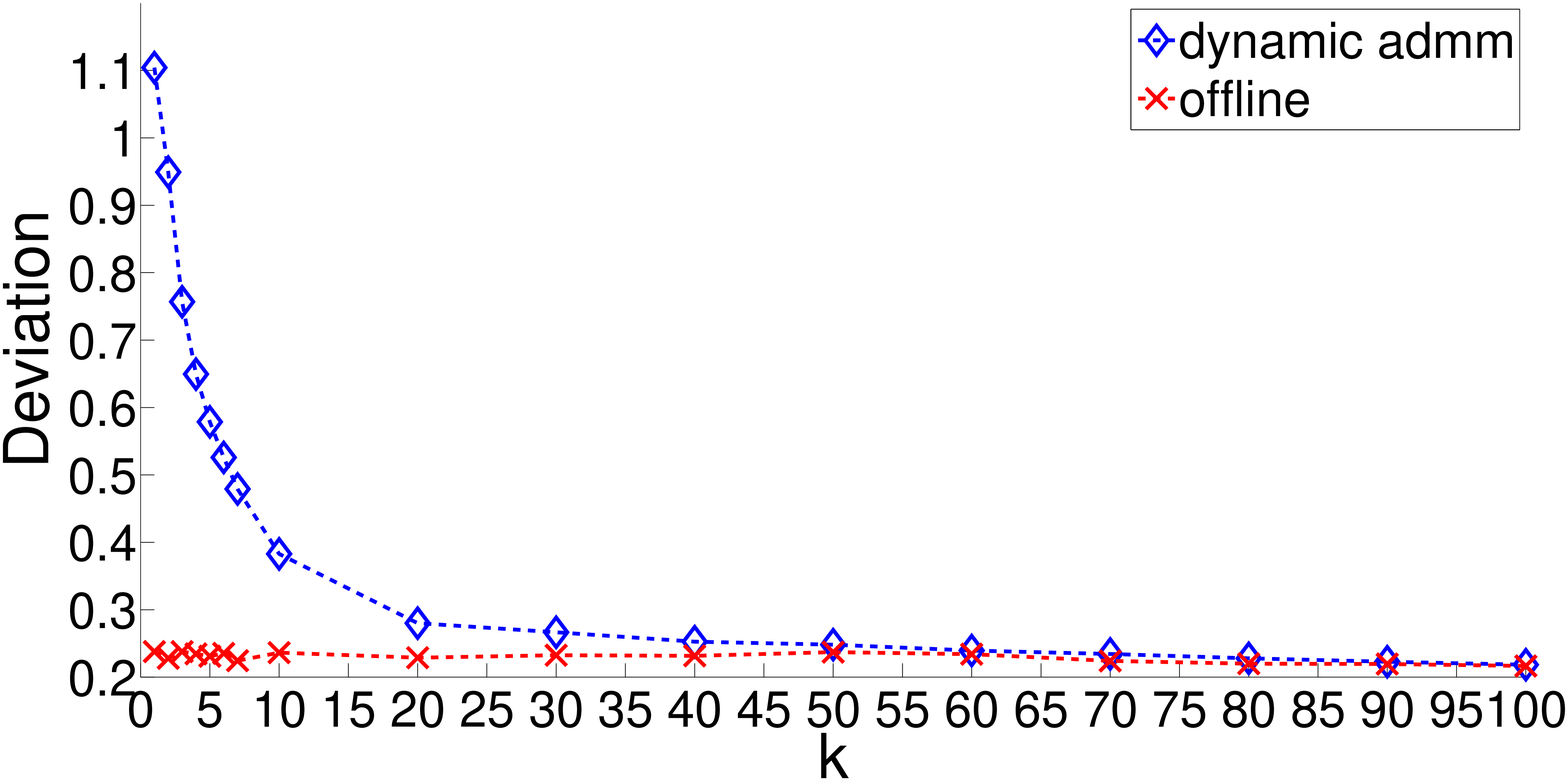}}
\subfigure[Fast time-variant case, i.e., $\eta=0.1$.]{
\includegraphics[scale = 0.3]{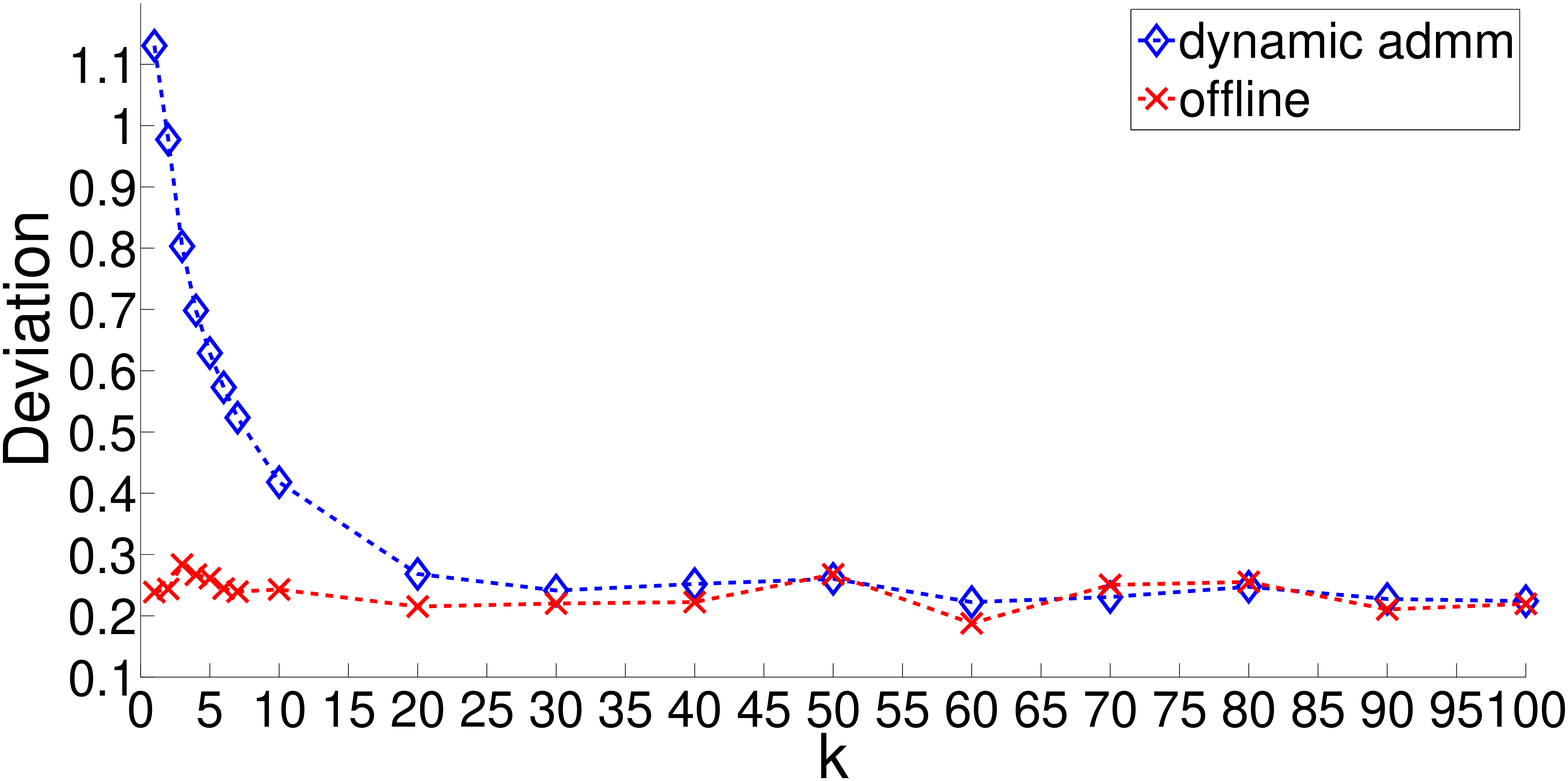}}

\caption{$\|\check{\mathbf{x}}_k\|_2$ and $\|\check{\mathbf{x}}_k^*\|_2$, the deviations of the dynamic ADMM and the offline optimizor from the true sparsity pattern.}
\label{lasso_zero}
\end{figure}

Since the ground-truth $\widetilde{\mathbf{x}}_k$ is sparse and the aim of LASSO is to promote sparsity, a critical performance metric of a solution is how well can it recover the true sparsity pattern. To this end, we compute $\|\check{\mathbf{x}}_k\|_2$ and $\|\check{\mathbf{x}}_k^*\|_2$. Here $\check{\mathbf{x}}_k$ denotes the subvector of $\mathbf{x}_k$ composed of those entries at positions $\{1,...,p\}\backslash\{j_1,...,j_q\}$, i.e., the positions at which the ground-truth is identically zero. Similar definition holds for $\check{\mathbf{x}}_k^*$. $\|\check{\mathbf{x}}_k\|_2$ and $\|\check{\mathbf{x}}_k^*\|_2$ characterize the deviations of the dynamic ADMM ($\mathbf{x}_k$) and the offline optimizor ($\mathbf{x}_k^*$) from the true sparsity pattern. The convergence curves of $\|\check{\mathbf{x}}_k\|_2$ and $\|\check{\mathbf{x}}_k^*\|_2$ are illustrated in Fig. \ref{lasso_zero}-(a) and Fig. \ref{lasso_zero}-(b) for the slowly time-variant case and the fast time-variant case, respectively. We remark that in both cases, after some 50 iterations, the sparsity pattern deviation of the dynamic ADMM is close to that of the offline optimizor. This demonstrates that the dynamic ADMM has the same capability of identifying the true sparsity pattern as the offline optimizor does while the former enjoys significant computational advantage over the latter.

\begin{figure}
  \centering
  \includegraphics[scale=.3]{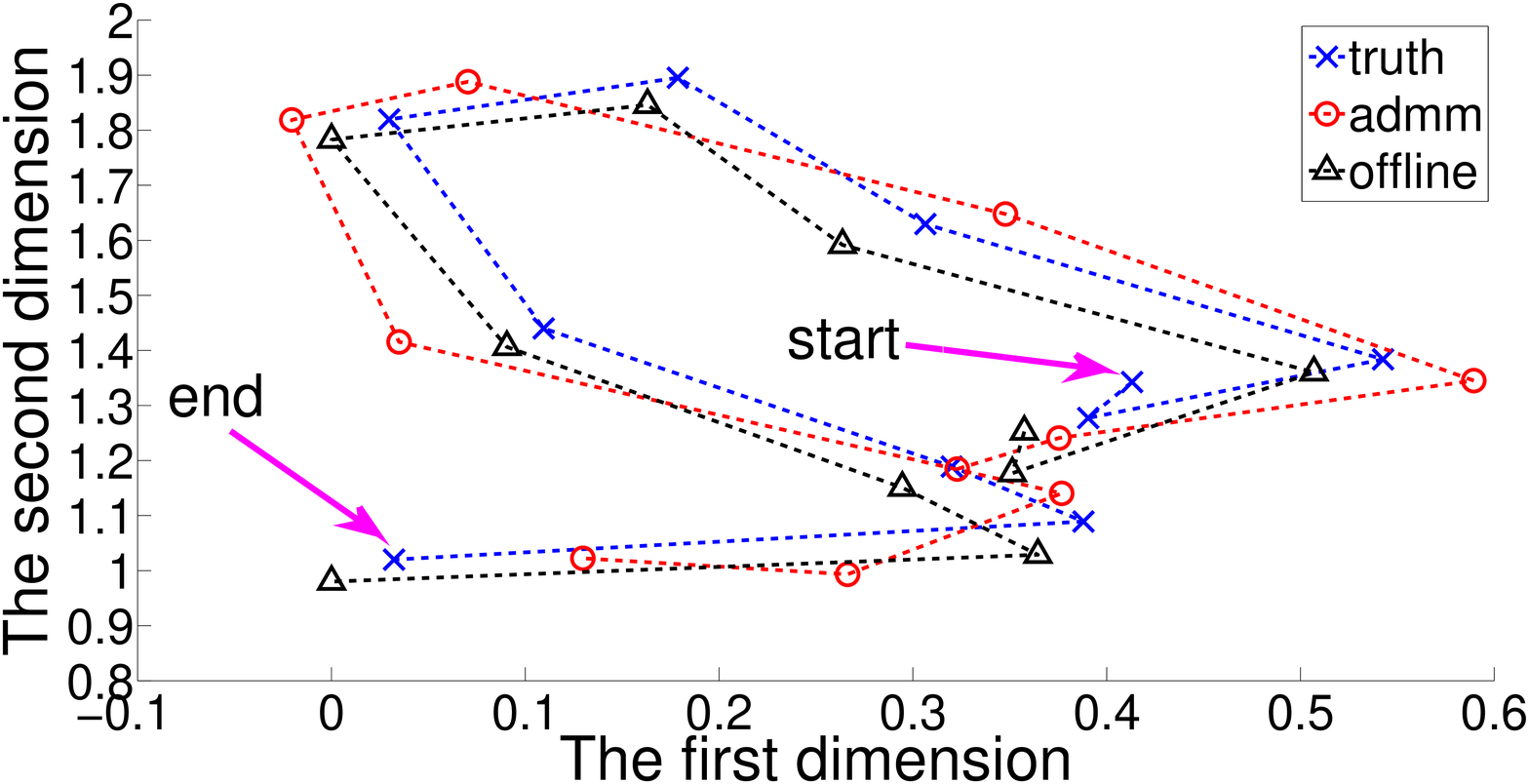}\\
  \caption{The trajectories of the two nonzero dimensions (i.e., $i_1,i_2$, corresponding to the horizontal axis and the vertical axis, respectively) of the dynamic ADMM, the offline optimizor and the ground-truth in one trial of the fast time-variant case}\label{lasso_track}
\end{figure}

Lastly, a more palpable result of the tracking performance is shown in Fig. \ref{lasso_track}, in which the trajectories of the two nonzero dimensions (i.e., $i_1,i_2$, corresponding to the horizontal axis and the vertical axis, respectively) of the dynamic ADMM, the offline optimizor and the ground-truth in one trial of the fast time-variant case are shown. The starting point corresponds to $k=10$ and the time gap between two adjacent points is 10. We observe that the dynamic ADMM can track the truth well. The tracking performance of the offline optimizor is somewhat better, but at the expense of its heavy or even intractable computational burden in many real-time applications.

\section{Conclusion}

In this paper, motivated by the dynamic sharing problem, we propose and study a dynamic ADMM algorithm, which can adapt to the time-varying optimization problems in an online manner. Theoretical analysis is presented to show that the dynamic ADMM converges linearly to some neighborhood of the optimal point. The size of the neighborhood depends on the inherent evolution speed, i.e., the drift, of the dynamic optimization problem across time: the more drastically the problem evolves, the bigger the size of the neighborhood. The impact of the drift on the steady state convergence behaviors of the dynamic ADMM is also investigated. Two numerical examples, namely a dynamic sharing problem and the dynamic LASSO, are presented to corroborate the effectiveness of the dynamic ADMM. We remark that the dynamic ADMM can track the time-varying optimal points quickly and accurately. For the dynamic LASSO, the dynamic ADMM has competitive performance compared to the benchmark offline optimizor while the former possesses significant computational advantage over the latter.

\bibliography{mybib}{}
\bibliographystyle{ieeetr}

\end{document}